\documentclass[12pt]{amsart}

\usepackage{amssymb}
\usepackage{amscd}
\usepackage{fullpage}
\usepackage{bbm}

\input xy
\xyoption{all}

\numberwithin{equation}{section}

\newtheorem{Proposition}[equation]{Proposition}
\newtheorem{Lemma}[equation]{Lemma}
\newtheorem{Theorem}[equation]{Theorem}
\newtheorem{Corollary}[equation]{Corollary}

\theoremstyle{definition}
\newtheorem{Definition}[equation]{Definition}
\newtheorem{Remark}[equation]{Remark}
\newtheorem{Example}[equation]{Example}

\def\sub{\subseteq}
\def\into{\hookrightarrow}
\def\onto{\twoheadrightarrow}
\def\iso{\cong}
\def\isoto{\stackrel{\sim}{\to}}
\def\bar{\overline}
\def\ch{\operatorname{ch}}
\def\ad{{\operatorname{ad}\,}}
\def\ne{{\operatorname{ne}}}
\def\Pr{\operatorname{Pr}}
\def\res{\operatorname{res}}

\def\gr{\operatorname{gr}}
\def\id{\operatorname{id}}
\def\Lie{\operatorname{Lie}}
\def\C{\mathbb C}
\def\R{\mathbb R}
\def\Z{{\mathbb Z}}
\def\1{\mathbbm 1}
\def\bu{\mathbf u}
\def\bv{\mathbf v}
\def\bw{\mathbf w}
\def\bx{\mathbf x}
\def\by{\mathbf y}
\def\bz{\mathbf z}
\def\End{{\operatorname{End}}}
\def\im{{\operatorname{im}}}
\def\Wh{\operatorname{Wh}}
\def\lan{\langle}
\def\ran{\rangle}
\def\ostar{\circledast}

\def\a{\mathfrak a}
\def\b{\mathfrak b}
\def\g{\mathfrak g}
\def\k{\mathfrak k}
\def\l{\mathfrak l}
\def\m{\mathfrak m}
\def\n{\mathfrak n}
\def\p{\mathfrak p}
\def\q{\mathfrak q}
\def\r{\mathfrak r}
\def\t{\mathfrak t}
\def\z{\mathfrak z}
\def\gl{\mathfrak{gl}}
\def\sl{\mathfrak{sl}}

\def\cO{\mathcal O}
\def\cS{\mathcal S}

\title[Translation for finite $W$-algebras]
{Translation for finite $W$-algebras}
\author{Simon M.~Goodwin}

\address{School of Mathematics, University of Birmingham, Birmingham, B15 3LX,~UK}
\email{goodwin@for.mat.bham.ac.uk}


\begin{document}

\begin{abstract}
A finite $W$-algebra $U(\g,e)$ is a certain finitely generated
algebra that can be viewed as the enveloping algebra of the Slodowy
slice to the adjoint orbit of a nilpotent element $e$ of a complex
reductive Lie algebra $\g$. It is possible to give the tensor
product of a $U(\g,e)$-module with a finite dimensional
$U(\g)$-module the structure of a $U(\g,e)$-module; we refer to such
tensor products as translations. In this paper, we present a number
of fundamental properties of these translations, which are expected
to be of importance in understanding the representation theory of
$U(\g,e)$.
\end{abstract}

\maketitle

\section{Introduction} \label{S:intro}

Let $\g$ be a reductive Lie algebra over $\C$ and let $e \in \g$ be
nilpotent. The finite $W$-algebra $U(\g,e)$ associated to the pair
$(\g,e)$ is a finitely generated algebra obtained from $U(\g)$ by a
certain quantum Hamiltonian reduction; for a definition of
$U(\g,e)$, we refer the reader to Section~\ref{S:finW}. Finite
$W$-algebras were introduced to the mathematical literature by
Premet in 2002, see \cite{Pr1}. A special case of the definition,
when there is an even good grading for $e$, first appeared in the
PhD thesis of Lynch \cite{Ly}, extending work of Kostant for the
case where $e$ is regular nilpotent \cite{Ko}. Since \cite{Pr1},
there has been a great deal of research interest in finite
$W$-algebras and their representation theory, see for example
\cite{Br,BGK,BK1,BK2,BK3,Gi,GRU,Lo1,Lo2,Lo3,Lo4,Pr2,Pr3,Pr4}. This
is largely due to close connections between the representation
theory of $U(\g,e)$ and that of $U(\g)$, which are principally
through Skryabin's equivalence, see \cite{Sk}. This is discussed
below and provides an important connection between the primitive
ideals of $U(\g)$ whose associated variety contains the adjoint
orbit of $e$, and the primitive ideals of $U(\g,e)$; see \cite[Thm.\
3.1]{Pr2}, \cite[Thm.\ 1.2.2]{Lo1} and \cite[Thm.\ 1.2.2]{Lo2}.

In mathematical physics, finite $W$-algebras and their affine
counterparts have attracted a lot of attention under a slightly
different guise; see for example \cite{BT,DK,VD}.  It is proved in
\cite{DDDHK} that the definition in the mathematical physics
literature via BRST cohomology agrees with Premet's definition,
\cite{Pr1}. The equivalence of the definitions is of great
importance in \cite{BGK}, and also plays a significant role here.

\medskip

For the remainder of the introduction $M$ is a finitely generated
$U(\g,e)$-module and $V$ is a finite dimensional $U(\g)$-module.  We
define the {\em translation $M \ostar V$ of $M$ by $V$} by
transporting the tensor product on $U(\g)$-modules through
Skryabin's equivalence; we refer the reader to Section~\ref{S:trans}
for a precise definition.   Such translations are expected to be of
importance in understanding the representation theory of
$U(\g,e)$.

In this paper we prove a number of properties of translations. A
number of our results are generalizations of results from \cite[Ch.\
8]{BK2} for the case $\g = \gl_n(\C)$, though we require different
methods in general.  We outline our main results and the structure
of the paper below.

\smallskip

After giving some preliminaries in Section~\ref{S:prelim}, we
consider both the Whittaker model definition of $U(\g,e)$ and its
definition via nonlinear Lie algebras in Section~\ref{S:finW}; the
latter is our preferred definition in the rest of the paper. We
recall the equivalence of these two definitions and also present
some structure theory of $U(\g,e)$.

In Section~\ref{S:trans},  we give the definition of translation for
both the Whittaker model definition and the definition via nonlinear
Lie algebras.  In Lemma~\ref{L:transequiv}, we show that these two
definitions of translation are equivalent.

The principal goal of Section~\ref{S:vspiso} is to prove that there
is an isomorphism of vector spaces
\begin{equation} \label{e:iso}
M \ostar V \iso M \otimes V,
\end{equation}
so that translations leads to the structure of a $U(\g,e)$-module on
$M \otimes V$. This isomorphism is a consequence of Theorem
\ref{T:iso}, which considers a certain (Kazhdan) filtration on $M
\ostar V$ and the associated graded module. Although Theorem
\ref{T:iso} implies existence of an isomorphism as in \eqref{e:iso},
it does not give an explicit isomorphism.  This is remedied is
\S\ref{ss:lift}, where we discuss explicit isomorphisms, which are
natural in both $M$ and $V$; we note, however, that these
isomorphism are not canonical.  In particular, we point the reader
to {\em lift matrices} in Definition~\ref{D:lift}, which are
remarkable matrices that allow one to describe the isomorphisms, and
are of great importance in the rest of the paper.

In \S\ref{ss:loop}, we recall the loop filtration on $U(\g,e)$ and
define a loop filtrations on $M$ and $M \ostar V$. The associated
graded algebra $\gr' U(\g,e)$ for the loop filtration is $U(\g^e)$,
where $\g^e$ is the centralizer of $e$ in $\g$. In Proposition
\ref{P:loop}, we prove that $\gr'(M \ostar V)$ and $\gr' M \otimes
V$ are isomorphic as $U(\g^e)$-modules.

In Section~\ref{S:basic}, we present some elementary properties of
translation.  In Proposition~\ref{P:tildepmod}, we give a tensor
identity for translations of certain $U(\g,e)$-modules that occur as
restrictions.  Then in Lemmas~\ref{L:assoc} and~\ref{L:adjdual}, we
show that translation is ``associative'', and that $? \ostar V^*$ is
biadjoint to $? \ostar V$, where $V^*$ denotes the dual
$U(\g)$-module.

In Section~\ref{S:hw}, we consider relationship between translations
and the highest weight theory from \cite[\S4]{BGK}. Our first main
result of this section is Proposition~\ref{P:transadm}, which says
that the category $\cO(e)$ of $U(\g,e)$-modules from
\cite[\S4.4]{BGK} is stable under translation; the category $\cO(e)$
is an analogue of the usual BGG category $\cO$ of $U(\g)$-modules.
Recently, in \cite{Lo3}, Losev has proved that the $\cO(e)$ is
equivalent to a certain category of Whittaker modules for $U(\g)$,
which in particular verified \cite[Conj.\ 5.3]{BGK}. We note that
this equivalence of categories enables an alternative definition of
translation for $M \in \cO(e)$, which seems likely to be equivalent.
The second main result in Section~\ref{S:hw} concerns translations
of Verma modules for $U(\g,e)$ as defined in \cite[\S4.2]{BGK}.  In
Theorem~\ref{T:filtverma}, we show that the translation of a Verma
module is filtered by Verma modules.

The BRST definition of $U(\g,e)$ is recalled in Section
\ref{S:brst}.   The definition of translation in the BRST setting is
given in Definition~\ref{D:BRSTtrans}, and shown to be equivalent to
the previous definition in Proposition~\ref{P:brstequiv}.  This
equivalence, and the explicit form of it given in
Theorem~\ref{T:brst} is of importance in Section~\ref{S:transdual}
as mentioned below.

Section~\ref{S:right} is a short technical section of the paper.
Throughout the paper, we work with ``left-handed'' definitions, but
in Section~\ref{S:transdual} it is necessary to consider
``right-handed'' versions of certain objects.  In Section
\ref{S:right}, we present the required definitions and right-handed
analogues of results.

The main result of Section~\ref{S:transdual} says that translation
commutes with duality in a certain sense.  The exact statement is
given in Theorem \ref{T:transdual}; in essence it says that
$$
\bar M \ostar \bar V \iso \bar{M \ostar V},
$$
where bars denote {\em restricted duals}.  This result is a
consequence of Theorem~\ref{T:dualizable}, which provides a striking
relationship between lift matrices for $V$ and $\bar V$.

The final section of this paper contains a slightly technical
result. The Whittaker model definition of $U(\g,e)$ depends on two
choices: of a good grading $\g = \bigoplus_{j \in \R} \g(j)$ for $e$
and an isotropic subspace $\l \sub \g(-1)$.  Thanks to a
construction of Gan and Ginzburg \cite[Thm.\ 4.1]{GG} it is known
that the definition does not depend on the choice of $\l$ up to
isomorphism. In \cite[Thm.\ 1]{BG}, it was shown that the definition
of $U(\g,e)$ does not depend on the choice of good grading up to
isomorphism.  In Proposition~\ref{P:indepiso} and
Theorem~\ref{T:indepgood}, we show that in the appropriate sense the
translation of $M$ by $V$ does not depend on the choice of $\l$ and
of the good grading.  This result justifies the fact that,
throughout most of the paper, we only consider translations for the
definition of $U(\g,e)$ via nonlinear Lie algebras and a fixed good
grading.

\smallskip

We end the introduction with a couple of remarks.  First, we note
that the definition of translations leads to a definition of
translation functors, in analogy to those in other setting: for the
case of reductive algebraic groups see \cite[II.2]{Ja}.  We do not
consider these functors in this paper, but remark here that an
alternative approach to translation functors for finite $W$-algebras
using the theory of Whittaker $\mathcal D$-modules is given in
\cite[\S5]{Gi}. The two approaches are expected to be related.

Second we comment on Losev's definition of $U(\g,e)$ via Fedosov
quantization, see \cite[\S3]{Lo1}. It is possible to define
translation with this definition of $U(\g,e)$ as in \cite[\S4]{Lo1}
and it is expected that this definition of translation is equivalent
to those considered in this paper.

\subsection*{Acknowledgments}
I am very grateful to Jonathan Brundan who suggested the topic of
this paper, and made a number of suggestions to improve the paper. I
would also like to thank Alexander Kleshchev for helpful
discussions.  This paper was partly written during a stay at the
Isaac Newton Institute for Mathematical Sciences, Cambridge during
the "Algebraic Lie Theory" Programme in 2009; I would like to thank
the institute and its staff for the hospitality.

\section{Preliminaries} \label{S:prelim}

Throughout this paper we work over the field of complex numbers $C$;
though all of our results remain valid over an algebraically closed
field of characteristic $0$. As a convention throughout this paper,
by a ``module'' we mean a finitely generated left module; we state
explicitly when we are considering right modules, which are also
always finitely generated.

\subsection{Notation} \label{ss:notn}

Let $G$ be connected reductive algebraic group over $\C$, let $\g$
be the Lie algebra of $G$.  Let $(\cdot|\cdot)$ be a non-degenerate
symmetric invariant bilinear form on $\g$.  For $x \in \g$, we write
$\g^x$ for the centralizer of $x$ in $\g$.  We write $\z(\g)$ for
the centre of $\g$.

Let $e$ be a nilpotent element of $\g$ and fix an ${\sl_2}$-triple
$(e,h,f)$.   Define the linear functional
$$
\chi:\g \rightarrow \C, \qquad x \mapsto (e|x).
$$
Let $\t^e$ be a maximal toral subalgebra of $\g^e \cap \g^h$, and
let $\t$ be a maximal toral subalgebra of $\g$ containing $\t^e$ and
$h$.  The root system of $\mathfrak{g}$ with respect to
$\mathfrak{t}$ is denoted by $\Phi$.

We recall that a $\Z$-grading
$$
\g = \bigoplus_{j \in \Z} \g(j)
$$
of $\g$ is called a {\em good grading} for $e$ if $e \in \g(2)$,
$\g^e \sub \bigoplus_{j \geq 0} \g(j)$ and $\z(\g) \sub \g(0)$, see
\cite{EK}. The standard example of a good grading for $e$ is the
Dynkin grading obtained by taking the $\ad h$-eigenspace
decomposition of $\g$.  We fix a good grading for the remainder of
the paper, which we may assume satisfies $f \in \g(-2)$ and $\t \sub
\g(0)$. In Section~\ref{S:indep}, we allow the grading to vary and
consider the more general notion of good $\R$-gradings.  One can
easily show that there exists $c \in \t$ such that the good grading
of is the $\ad c$-eigenspace decomposition, i.e.\ $\g(j) = \{x \in
\g \mid [c,x] = jx\}$.

The vector space $\g(-1)$ is denoted by $\k$. Let $\omega = \lan
\cdot | \cdot \ran$ be the non-degenerate alternating form on $\k$
defined by
$$
\lan x|y \ran = \chi([y,x]).
$$
Let $\p = \bigoplus_{j \geq 0} \g(j)$; this is a parabolic
subalgebra of $\g$. We abbreviate $\n = \bigoplus_{j < 0} \g(j)$,
which is a nilpotent subalgebra of $\g$, and fix a basis
$b_1,\dots,b_r$ for $\n$ such that $b_i$ is weight vector for $\t$
with weight $\beta_i \in \Phi$, and $b_i \in g(-d_i)$ with $d_i \in
\Z_{\ge 1}$. We write $f_1,\dots,f_r$ for the dual basis of $\n^*$.

In the sequel we require ``copies'' of $\n$ and $\k$ given by
$\n^{\ch} = \{x^{\ch} \mid x \in \n\}$ and $\k^{\ne} = \{x^{\ne}
\mid x \in \k\}$ respectively. Given $x \in \g$, we may write $x =
\sum_{j \in \Z} x(j)$ for the decomposition of $x$ with respect to
the good grading, i.e.\ $x(j) \in \g(j)$ for each $j$. In some
situations we wish to make sense of $x^{\ch}$, when $x \in \g$ but
$x \notin \n$, by convention we set $x^{\ch} = x(<0)^{\ch}$, where
$x(<0) = \sum_{j < 0} x(j)$; there is an analogous convention for
$x^{\ne}$ when $x \in \g$ but $x \notin \k$, i.e.\ $x^{\ne} =
x(-1)^{\ne}$.

\subsection{Recollection on non-linear Lie algebras} \label{ss:nlla}

In this article we use an easy special case of
the notion of a non-linear Lie superalgebra from \cite[Defn.\
3.1]{DK}.

In this article, a {\em non-linear Lie superalgebra} means a vector
superspace $\a = \a_{\bar 0} \oplus \a_{\bar 1}$ equipped with a {\em
non-linear Lie bracket} $[\cdot,\cdot]$: that is,
a parity preserving linear map
$\a \otimes \a \rightarrow T(\a)$ ($=$ the tensor algebra on $\a$)
satisfying the following conditions for all homogeneous $a,b,c \in
\a$:
\begin{enumerate}
\item[(i)]
$[a,b] \in \C \oplus \a$;
\item[(ii)] $[a,b] = (-1)^{p(a) p(b)} [b,a]$, where $p(a) \in \Z_2$
denotes parity; and
\item[(iii)] $[a,[b,c]] = [[a,b],c]+(-1)^{p(a) p(b)} [b,[a,c]]$
(interpreted using the convention that any bracket with a scalar is
zero).
\end{enumerate}
This definition agrees with the general notion of non-linear Lie
superalgebra from \cite[Defn.\ 3.1]{DK} when the grading on $\a$ in
the general setup is concentrated in degree $1$.

The {\em universal enveloping superalgebra} of a non-linear Lie
superalgebra $\a$ is defined to be $U(\a) = T(\a)/M(\a)$, where
$M(\a)$ is the two-sided ideal of $T(\a)$ generated by the elements
$a \otimes b - (-1)^{p(a) p(b)} b \otimes a - [a,b]$ for all
homogeneous $a,b \in \a$. By a special case of \cite[Theorem
3.3]{DK}, $U(\a)$ is {\em PBW generated} by $\a$ in the sense that
if $\{x_1,\dots,x_m\}$ is any homogeneous ordered basis of $\a$ then
the ordered monomials
$$
\{x_1^{a_1} \cdots x_m^{a_m} \mid a_i \in \Z_{\ge 0}
\text{ if } p(x_i) =  \bar 0
\text{ and } a_i \in \{0,1\} \text{ if } p(x_i) =  \bar 1\}
$$
give a basis for $U(\a)$.

By a {\em subalgebra} of a non-linear Lie superalgebra $\a$ we mean
a $\Z_2$-graded subspace $\b$ of $\a$ such that $[\b,\b] \subseteq
\C \oplus \b$. In that case $\b$ is itself a non-linear Lie
superalgebra and $U(\b)$ is identified with the subalgebra of
$U(\a)$ generated by $\b$.

We call $\a$ a {\em non-linear Lie algebra} if it is purely even.

\section{Finite W-algebras} \label{S:finW}

In this section we give both the Whittaker model definition of the
finite $W$-algebra associated to $e$ denoted $W_\l$, and the
definition via non-linear Lie algebras denoted $U(\g,e)$.  Then we
recall the equivalence of these definitions (for the case $\l = 0$)
from \cite[\S2]{BGK}.  The definition via non-linear Lie algebras is
the preferred formulation in most of the paper, but the Whittaker
model definition is required for  Theorem~\ref{T:iso}, which is a
fundamental result in this paper.  We also present some results on
finite $W$-algebras that are required in the sequel; these are
contained in \S\ref{ss:structure} and \S\ref{ss:nllaskryb}.

\subsection{Whittaker model definition} \label{ss:whittw}
Before we define $W_\l$ we introduce some notation. We choose an
isotropic subspace $\l$ of $\g(-1)$ with respect to the alternating
form $\omega = \lan \cdot | \cdot \ran$ on $\k = \g(-1)$.   The
annihilator of $\l$ with respect to $\omega$ is $\l^{\perp_\omega} =
\{x \in \k \mid \lan x|y \ran = 0 \text{ for all } y \in \l\}$.
Define the nilpotent subalgebras
$$
\m_\l = \l \oplus \bigoplus_{j < -1} \g(j) \quad \text{and} \quad
\n_\l = \l^{\perp_\omega} \oplus \bigoplus_{j < -1} \g(j).
$$

Let $I_\l$  be the left ideal of $U(\g)$ generated by $\{x-\chi(x)
\mid x \in \m_\l\}$, and define the left $U(\g)$-module $Q_\l =
U(\g)/I_\l$.  The adjoint action of $\n_\l$ on $U(\g)$ induces a
(well-defined) adjoint action on $Q_\l$. The {\em Whittaker model
definition of the finite $W$-algebra} associated to $\g$ and $e$ is
$$
W_\l = H^0(\n_\l,Q_\l) = Q_\l^{\n_\l},
$$
where the Lie algebra cohomology is taken with respect to adjoint
action of $\n_\l$ on $Q_\l$.  More explicitly, $W_\l$ is the space
of {\em twisted $\n_\l$-invariants}:
$$
W_\l = \{u + I_\l \in Q_\l \mid [x,u] \in I_\l \text{ for all } x
\in \n_\l\}.
$$
It is easy to check that multiplication in $U(\g)$ gives rise to a
well-defined multiplication on $W_\l$. We note that the definition
of $W_\l$ depends on the choice of $\l$ and on the choice of good
grading; this is discussed in Section~\ref{S:indep} where we recall
the proof from \cite[\S5.5]{GG} that $W_\l$ is independent of $\l$
up to isomorphism, and the proof of independence of good grading
from \cite[Thm.\ 1]{BG}.

In the sequel, we denote $\1_\l = 1 + I_\l \in Q_\l$, so then we
have $u\1_\l = u + I_\l$ for $u \in U(\g)$.  There is the
$U(\g)$-$W_\l$-bimodule structure on $Q_\l$: the right action of
$W_\l$ being given by $(u\1_\l)(v\1_\l) = uv\1_\l$, it is
straightforward to check that this is well-defined.  For the case
$\l = 0$, we abbreviate notation and write $I = I_\l$, $Q =
Q_\l$, $W = W_\l$ and $\1 = \1_\l$.

We now introduce some notation so that we can discuss the main
structure theorem for $W_\l$; this is required for the proof of
Theorem~\ref{T:iso}, which is a fundamental result for this paper.
The {\em Kazhdan filtration} of $U(\g)$ is defined by declaring that
$x \in \g(j)$ has Kazhdan degree $j + 2$.  The Kazhdan filtration
induces filtrations on both $Q_\l$ and $W_\l$.
As is shown in \cite[\S4]{GG},
the associated graded module $\gr Q_\l$ of $Q_\l$ can be identified with the
coordinate algebra $\C[e + \m_\l^\perp]$, where $\m_\l^\perp$
denotes the annihilator of $\m_\l$ in $\g$ with respect to the form
$(\cdot|\cdot)$.

Let $N_\l$ be the unipotent subgroup of $G$ corresponding to
$\n_\l$. The affine space $\cS = e + \g^f$ is called the {\em
Slodowy slice} to the nilpotent orbit of $e$; it is a transverse
slice to the $G$-orbit of $e$. By \cite[Lem.\ 2.2]{GG}, there is an
isomorphism of varieties $N_\l \times (e + \g^f) \isoto e +
\m_\l^\perp$, given by the adjoint action map.  As a consequence, we
obtain the identification $\gr Q_\l \iso \C[N_\l] \otimes \C[\cS]$.
The main structure theorem for $W_\l$ says that
\begin{equation} \label{e:filt}
\gr W_\l \iso \C[e+\m_\l^\perp]^{N_\l} \iso \C[\cS],
\end{equation}
where $\gr W_\l$ is the associated graded algebra of $W_\l$ with
respect to the Kazhdan filtration. This was first proved by Premet
in \cite[Thm.\ 4.6]{Pr1}; the approach followed here is that given
by Gan and Ginzburg, \cite[Thm.\ 4.1]{GG}

\subsection{Definition via non-linear Lie algebras} \label{ss:nllaw}

We begin by defining the non-linear Lie algebra $\tilde \g$.  Recall
that $\k^{\ne}$ is a ``copy'' of $\k$.  We give $\k^{\ne}$ the
structure of a nonlinear Lie algebra with bracket defined by
$[x^{\ne},y^{\ne}] = \lan x | y\ran$. The nonlinear Lie algebra
$\tilde \g = \g \oplus \k^{\ne}$ is defined by extending the bracket
on $\g$ and $\k^{\ne}$ and declaring that $[x,y^{\ne}] = 0$ for $x
\in \g$ and $y \in \k$. We define the subalgebra $\tilde \p = \p
\oplus \k^{\ne}$ of $\tilde \g$.  Note that as $\g$ commutes with
$\k^{\ne}$, we have tensor decompositions $U(\tilde \g) \iso U(\g)
\otimes U(\k^{\ne})$ and $U(\tilde \p) \iso U(\p) \otimes
U(\k^{\ne})$.  Further, $U(\k^{\ne})$ is isomorphic to the Weyl
algebra associated to $\k$ and the form $\omega$.

We define $\tilde I$ to be the left ideal of $U(\tilde \g)$
generated by $x - x^{\ne} -\chi(x)$ for $x \in \n$.  We define $\tilde Q$
to be the $U(\tilde \g)$-module $U(\tilde \g)/\tilde I$, and
denote $\tilde \1 = 1 + \tilde I \in \tilde Q$.  By the PBW
theorem for $U(\tilde \g)$ we have a direct sum decomposition
$U(\tilde \g) = U(\tilde \p) \oplus \tilde I$, so we can
identify $\tilde Q \iso U(\tilde \p)$ as vector spaces.  We write $\Pr :
U(\tilde \g) \to U(\tilde \p)$ for the projection along the above direct
sum decomposition.  There is an action of
$\n$ on $U(\tilde \p)$ by
\begin{equation} \label{e:praction}
x \cdot u = \Pr((x-x^\ne-\chi(x))u),
\end{equation}
which gives $U(\tilde \p)$ the structure of an $\n$-module;
under the identification of vector spaces $\tilde Q \iso U(\tilde \p)$
this coincides with the action of $\n$ on $\tilde Q$
by $x - x^\ne - \chi(x)$.  We note
that this action is the same as the {\em twisted adjoint action} of
$\n$ on $U(\tilde \p)$ given by
$$
x \cdot u = \Pr([x-x^\ne,u]).
$$
We define the {\em finite $W$-algebra}
$$
U(\g,e) = H^0(\n,U(\tilde \p)) = U(\tilde \p)^\n,
$$
where the cohomology is taken with respect to the action of $\n$ given in
\eqref{e:praction}. More explicitly, we have
$$
U(\g,e) = \{u \in U(\tilde \p) \mid \Pr([x-x^{\ne},u]) = 0 \text{
for all } x \in \n\}
$$
is the space of twisted $\n$-invariants in $U(\tilde \p)$. It is a
subalgebra of $U(\tilde \p)$, see \cite[Thm.\ 2.4]{BGK}.

There is an right action of $U(\g,e)$ on $\tilde Q$ making $\tilde
Q$ into a $U(\tilde \g)$-$U(\g,e)$-bimodule.  This action is given
by $(u \tilde \1)v  = (uv)\tilde \1$ for $u \in U(\tilde \p)$ and $v
\in U(\g,e)$, i.e.\ it is given by multiplication in $U(\tilde \p)$
under the identification $U(\tilde \p) \iso \tilde Q$.

We now recall the isomorphism between $U(\g,e)$ and $W$
given in \cite[Lem.\ 2.3]{BGK}, see also \cite[\S2.4]{Pr2}.
There is a well defined action of $U(\tilde
\g)$ on $Q$ given by extending the regular action of $U(\g)$ and
defining $x^\ne \cdot (u + I) = ux + I$ for $x \in \k$, and $u \in
U(\g)$. By \cite[Lem.\ 2.3]{BGK}, the natural map $U(\tilde \g) \to
Q$ given by $u \mapsto u \cdot (1+I)$ intertwines the twisted
adjoint action of $\n$ on $U(\tilde \p)$ with the adjoint action of
$\n$ on $Q$. This is used to prove \cite[Thm.\ 2.4]{BGK}, which we
state below for convenience of reference.

\begin{Lemma} \label{L:equiv}
The natural map $U(\tilde \g) \to Q$ given by $u \mapsto u \cdot
\1$ restricts to an isomorphism of vector spaces $U(\tilde \p)
\isoto Q$ and an isomorphism of algebras $U(\g,e) \isoto W$.
\end{Lemma}

We next discuss the Kazhdan filtration on $U(\g,e)$ recalling the
required parts of the discussion in \cite[\S3.2]{BGK}.  The Kazhdan
filtration is extended to $U(\tilde \g)$ by saying that $x^\ne \in
\k^\ne$ has degree $1$.  Then $U(\tilde \p)$ inherits a non-negative
filtration such that the associated graded algebra $\gr U(\tilde \p)
\iso S(\tilde \p)$, where the symmetric algebra $S(\tilde \p)$ has
the Kazhdan grading. The twisted adjoint action of $\n$ on $U(\tilde
\p)$ induces a graded action of $\n$ on $S(\tilde \p)$, and through
the isomorphism in Lemma~\ref{L:equiv} and \eqref{e:filt}, we get
$\gr U(\g,e) \iso H^0(\n,S(\tilde \p)) = S(\tilde \p)^\n$, where the
cohomology is taken with respect to this action.  By \cite[Lem.\
2.2]{BGK}, there is a direct sum decomposition
\begin{equation} \label{e:dirsum}
\tilde \p = \g^e \oplus \bigoplus_{j \ge 2} [f,\g(j)] \oplus \k^\ne.
\end{equation}
The projection $\tilde \p \onto \g^e$ along this decomposition
induces a homomorphism
$\zeta : S(\tilde \p) \to S(\g^e)$.
Then \cite[Lem.\ 3.5]{BGK} says that $\zeta$ restricts to an
isomorphism
\begin{equation} \label{e:zeta}
\zeta : S(\p)^\n \isoto S(\g^e).
\end{equation}

\subsection{Some structure theory of $U(\g,e)$} \label{ss:structure}

In this subsection we present some results regrading the structure
of $U(\g,e)$ that are required in the sequel.

\smallskip

First we apply the discussion of the Kazhdan filtration in the
previous subsection to show, in Lemma~\ref{L:free}, that $\tilde Q$
is free as a right $U(\g,e)$-module; this a slight generalization of
part (3) of the theorem in \cite{Sk}.  We set $\r = \bigoplus_{j \ge
2} [f,\g(j)] \oplus \k^\ne$

\begin{Lemma} \label{L:free} $ $
\begin{enumerate}
\item[(i)] $\tilde Q$ is free as a right $U(\g,e)$-module.
\item[(ii)] A basis of $\tilde Q$
as a free right $U(\g,e)$-module can be constructed as follows.
Choose a basis $(x_1^{\ne}),\dots,(x_{2s}^{\ne}),
x_{2s+1},\dots,x_r$ of $\r$. Then the monomials
$$
(x_1^{\ne})^{a_1}
\dots (x_{2s}^{\ne})^{a_{2s}} x_{2s+1}^{a_{2s+1}} \dots x_r^{a_r}
\tilde \1 \in \tilde Q
$$
 with $a_i \in \Z_{\ge 0}$ form a basis for
$\tilde Q$ as a free right $U(\g,e)$-module.
\end{enumerate}
\end{Lemma}

\begin{proof}
From \eqref{e:dirsum} we get $\gr \tilde Q \iso S(\r) \otimes
S(\g^e)$. We also have the isomorphism $\gr U(\g,e) \isoto S(\g^e)$
given by the restriction of $\zeta$ from \eqref{e:zeta}.  Now an
induction on Kazhdan degree shows that $\gr \tilde Q \iso S(\r)
\otimes \gr U(\g,e)$, so that $\gr \tilde Q$ is a free right $\gr
U(\g,e)$-module.  As the Kazhdan filtration on $\tilde Q$ is
non-negative, a standard filtration argument shows that $\tilde Q$
is free as a right $U(\g,e)$-module; and moreover, if $\{\gr u_i
\tilde \1 \mid i \in J\}$, where $J$ is some indexing set, is a
basis of $\gr Q_\l$ as a free right $\gr U(\g,e)$-module, then
$\{u_i \tilde \1 \mid i \in J\}$ is a basis of $\tilde Q$ as a free
right $U(\g,e)$-module.  Now the above tensor decomposition $\gr
\tilde Q \iso S(\r) \otimes \gr U(\g,e)$ tells us that if
$(x_1^{\ne}),\dots,(x_{2s}^{\ne}), x_{2s+1},\dots,x_r$ is a basis of
$\r$, then the monomials $\gr (x_1^{\ne})^{a_1} \dots
(x_{2s}^{\ne})^{a_{2s}} x_{2s+1}^{a_{2s+1}} \dots x_r^{a_r} \tilde
\1 \in \gr \tilde Q$ form a basis for $\gr \tilde Q$ as a free right
$\gr U(\g,e)$-module.
\end{proof}

We now explain one way to construct a basis
$(x_1^{\ne}),\dots,(x_{2s}^{\ne}), x_{2s+1},\dots,x_r$ of $\r$ as in
the above lemma. Take the basis $b_1,\dots,b_r$ of $\n$, which we
recall consists of $\t$-weight vectors with $b_i \in \g(-d_i)$ and
$d_i \in \Z_{\ge 1}$; then an easy consequence of
$\sl_2$-representation theory is that for each $i$ there exists
unique $x_i \in \g(d_i-2)$ with $(x_i|[b_j,e]) = \delta_{ij}$ and
$(x_i|y) = 0$ for all $y \in \g^f$ (we assume that $b_1,\dots,b_r$
is ordered so that $x_1,\dots,x_{2s} \in \k$ and $x_{2s+1},\dots,x_r
\in \p$).

Using this basis of $\r$ , we may define the projection
\begin{equation} \label{e:chi}
\chi : \tilde Q \to U(\g,e)
\end{equation}
by $\chi((x_1^\ne)^{a_1} \dots (x_{2s}^\ne)^{a_{2s}}
x_{2s+1}^{a_{2s+1}}\dots x_r^{a_r} \tilde \1)= 0$ if $a_i \neq 0$
for some $i$, and $\chi(\tilde \1) = 1$.  We identify the associated
graded module of $\tilde Q$ with $S(\tilde \p)$, and write
\begin{equation} \label{e:eta}
\eta = \gr \chi : S(\tilde \p) \to S(\tilde \p)^\n
\end{equation}
for the associated graded map.  We record the following technical
lemma that we require in Section~\ref{S:vspiso}, it is a consequence
of \cite[Lem.\ 3.7]{BGK}.

\begin{Lemma} \label{L:zetaeta}
Let $\zeta$ be as in \eqref{e:zeta} and $\eta$ as in \eqref{e:eta}.
Then the following diagram commutes
$$
\xymatrix{ S(\tilde \p) \ar[r]^{\eta} \ar[dr]_{\zeta} & S(\tilde
\p)^\n \ar[d]^\zeta \\
& S(\g^e)}.
$$
\end{Lemma}

\smallskip

Next we recall another filtration of $U(\g,e)$ defined in
\cite[\S3.3]{BGK}; this filtration is called the good filtration in
{\em loc.\ cit.}, but we choose to use the terminology loop
filtration here, as in \cite[\S2]{BK1}. The good grading on $\g$
induces a grading of $U(\p)$, which we extend to a grading of
$U(\tilde \p)$ by declaring that elements of $\k^\ne$ have degree
$0$. Then $U(\g,e)$ is not in general a graded subalgebra of
$U(\tilde \p)$ but there is an induced filtration $(F_j' U(\g,e))_{j
\in \Z_{\ge 0}}$ of $U(\g,e)$ called the {\em loop filtration}. The
associated graded algebra $\gr' U(\g,e)$ is identified with a
subalgebra of $U(\tilde \p)$; in order to explicitly describe this
subalgebra we need to give some notation.

Let $z_1,\dots,z_{2s}$ be a symplectic basis for $\k$,
so that $\lan z_i | z_j^* \ran = \delta_{i,j}$ for all $1\leq
i,j\leq 2s$ where
$$
z_j^* := \left\{
\begin{array}{rl}
z_{j+s}&\text{for $j=1,\dots,s$,}\\
-z_{j-s}&\text{for $j=s+1,\dots,2s$.}
\end{array}\right.
$$
The Lie algebra homomorphism $\theta : \g^e \into U(\tilde \p)$ is
defined in \cite[Thm.\ 3.3]{BGK} by
\begin{equation} \label{e:theta0}
\theta(x) =
\left\{
\begin{array}{ll}
x + {\textstyle \frac{1}{2}} \sum_{i=1}^{2s} [x, z_i^*]^\ne z_i^\ne
&\text{if $x \in \g^e(0)$,}\\
x&\text{otherwise,}
\end{array}\right.
\end{equation}
which restricts to a Lie algebra homomorphism $\g^e(0) \into
U(\g,e)$.  We can extend $\theta$ to an algebra homomorphism $\theta
: U(\g^e) \to U(\tilde \p)$, and note that $U(\g^e)$ is graded from
the good grading on $\g$. Then \cite[Thm.\ 3.8]{BGK} says that
$\theta$ gives a $\t^e$-equivariant graded algebra isomorphism
\begin{equation} \label{e:theta}
\theta : U(\g^e) \isoto \gr' U(\g,e).
\end{equation}

\smallskip

Lastly in this subsection, we summarize \cite[Thm.\ 3.6 and Lem.\
3.7]{BGK}, which gives an explicit description of the structure of
$U(\g,e)$; we note that \cite[Thm.\ 3.6]{BGK} is essentially
\cite[Theorem 4.6]{Pr1}.

To begin, we note that \eqref{e:zeta} implies that there exists a
(non-unique) linear map
\begin{equation} \label{e:Theta}
\Theta:\g^e\into U(\g,e)
\end{equation}
such that $\Theta(x) \in F_{j+2} U(\g,e)$ and $\zeta(\gr_{j+2}
\Theta(x)) = x$ for each $x \in \g^e(j)$. We can choose $\Theta$ so
that it is $\t^e$-equivariant with respect to the embedding of
$\t^e$ in $U(\g,e)$ given by $\theta$; and such that $\gr' \Theta(x)
= \theta(x)$ for each $x \in \g^e(j)$ and $\Theta(t) = \theta(t)$
for each $t \in \g^e(0)$.

Now let $x_1,\dots,x_t$ is a basis of $\g^e$ that is homogeneous
with respect to the good grading and consists of $\t^e$-weight
vectors; say $x_i \in \g^e(n_i)$ with $\t^e$-weight $\gamma_i$. Then
for $j \ge 0$ the monomials
$$
\{\Theta(x_1)^{a_1} \dots \Theta(x_t)^{a_t} \mid a_i \in \Z_{\ge 0},
{\textstyle \sum_{i=1}^t} a_i (n_i+2) \le j \}
$$
form a basis for $F_j U(\g,e)$.

\subsection{Version of Skryabin's equivalence} \label{ss:nllaskryb}

We give a version of Skryabin's equivalence to obtain an equivalence
from the category $U(\g,e)$\-mod of $U(\g,e)$-modules to the
category $\Wh(\tilde \g,e)$ of {\em generalized Whittaker $U(\tilde
\g)$-modules} defined below. This equivalence is required in the
sequel; more precisely it is needed for the definition of
translation in \S\ref{ss:nllatrans} and we require
Corollary~\ref{C:vanish} in \S\ref{ss:brsttrans}. The proof of the
equivalence follows the same lines as that given for Skryabin's
equivalence in \cite[Thm.\ 6.1]{GG}.

Given a $U(\tilde \g)$-module $E$, there is an action of $\n$ on
$E$ defined by
\begin{equation} \label{e:dot}
x \cdot m = (x - x^\ne - \chi(x))m
\end{equation}
and henceforth referred to as the {\em dot action}.
We call $E$ a {\em generalized
Whittaker module} (with respect to $e$ and $\n$) if $x - x^\ne -
\chi(x)$ acts locally nilpotently on $E$ for all $x \in \n$, i.e.\
each $x \in \n$ acts locally nilpotently in the dot action. We write
$\Wh(\tilde \g,e)$ for the category of {\em generalized Whittaker
$U(\tilde \g)$-modules}.

Recall the $U(\tilde \g)$ module $\tilde Q$ from \S\ref{ss:nllaw}, it is
clear that this is a generalized Whittaker module.  It is also a right
$U(\g,e)$-module as explained in \S\ref{ss:nllaw}.
Thus there is a functor
$$
\tilde Q \otimes_{U(\g,e)} ? : U(\g,e)\-mod \to \Wh(\tilde
\g,e).
$$
There is also a functor
$$
H^0(\n,?): \Wh(\tilde \g,e) \to \tilde U(\g,e)\-mod,
$$
where the cohomology is taken with respect to the dot action of
$\n$; so that we have
$$
H^0(\n,E) = \{m \in E \mid (x-x^{\ne}-\chi(x))m = 0 \text{ for all }
x \in \n\}.
$$
It is straightforward to check that $H^0(\n,E)$ is a well-defined
$U(\g,e)$-module where the action is given by restricting the
$U(\tilde \g)$ action on $E$.

We are now in a position to state our version of Skryabin's
equivalence.  We only show how one can apply the proof of
\cite[Thm.\ 6.1]{GG}.

\begin{Theorem} \label{T:skryabin}
The functors $\tilde Q \otimes_{U(\g,e)} ?$ and $H^0(\n,?)$ are
quasi-inverse equivalences of categories.
\end{Theorem}

\begin{proof}
For $M \in U(\g,e)\-mod$, we have a natural map
$\phi : M \to H^0(\n,\tilde Q \otimes_{U(\g,e)} M)$ given by
$\phi(m) = \tilde \1 \otimes m$.
Using Lemma~\ref{L:equiv} we may identify $\tilde Q$ with $Q$ as a
filtered vector space and $U(\g,e)$ with $W$.  The dot action of
$\n$ on $\tilde Q \otimes_{U(\g,e)} M$ is same as the
twisted adjoint action given by $x \cdot (u\tilde \1 \otimes m) =
[x - x^{\ne},u]\tilde \1 \otimes m$, for $x \in \n$.
Under the identification of modules $\tilde Q \otimes_{U(\g,e)}
M \iso Q \otimes_W M$, the discussion before Lemma~\ref{L:equiv}
tells us that the dot action of $\n$ on $U(\tilde \p)
\otimes_{U(\g,e)} M$ is identified with the adjoint action of $\n$
on $Q \otimes_W M$. Now one can now prove that $\phi$ is an
isomorphism as in \cite[Thm.\ 6.1]{GG}.

For $E \in \Wh(\tilde \g)$, there is a natural map
$f: \tilde Q \otimes_{U(\g,e)} H^0(\n_\l,E) \to E$ given by
$f(u\tilde \1 \otimes v) = uv$.   The arguments in the proof of
\cite[Thm.\ 6.1]{GG} apply in our situation to prove that $f$ is an
isomorphism.
\end{proof}

Once we have the identifications $\tilde Q \iso Q$ and $U(\g,e)
\iso W$ used in the proof of Theorem~\ref{T:skryabin}, one can
deduce from the proof of \cite[Thm.\ 6.1]{GG} that $H^i(\n,\tilde Q
\otimes_{U(\g,e)}  M) = 0$ for $i > 0$, where the cohomology is
taken with respect to the dot action of $\n$. Therefore, we obtain
the following corollary, which we require in \S\ref{ss:brsttrans};
it is a slight generalization of part (4) of the theorem in \cite{Sk}.

\begin{Corollary} \label{C:vanish}
Let $E \in \Wh(\tilde \g,e)$.  Then $H^i(\n,E) = 0$ for $i > 0$,
where the cohomology is taken with respect to the dot action of
$\n$.
\end{Corollary}

\section{Definition of translation} \label{S:trans}

In this section we give the definition of translation in both the
Whittaker model definition and in the definition via non-linear Lie
algebras.  In Lemma~\ref{L:transequiv}, we prove that these
definitions are equivalent (for $\l =0$) through the isomorphism in
Lemma~\ref{L:equiv}.  We note that the definition of translation of
$W_\l$-modules depends on the choice of $\l$, and on the choice of
good grading.  In Section~\ref{S:indep}, we show that in fact the
definition is independent up to isomorphism in the appropriate
sense, which justifies us considering only translations of
$U(\g,e)$-modules in the rest of the paper.

\subsection{Translation of $W_\l$-modules} \label{ss:whitttrans}
The definition of translation for $W_\l$-modules given below
is a generalization of the definition given in \cite[\S8.2]{BK2}
for the case when $\g = \gl_n$ and the good grading for $e$ is even.

Let $M$ be a $W_\l$-module and let $V$ be a finite dimensional
$U(\g)$-module.  As explained in \S\ref{ss:whittw},
 $Q_\l$ has the structure of a
right $W_\l$-module and the tensor product $Q_\l \otimes_{W_\l} M$
is a $U(\g)$-module.  Further, there is an (adjoint) action of
$\n_\l$ on $Q_\l \otimes_{W_\l} M$ given by
$$
x \cdot (u\1_\l \otimes m) = [x,u]\1_\l \otimes m,
$$
for $x \in \n$, $u \in U(\g)$ and $m \in M$;
it is straightforward to check that this action is well defined and
gives $Q_\l \otimes_{W_\l} M$ the structure of an $\n_\l$-module.
Therefore, the tensor product $(Q_\l \otimes_{W_\l} M) \otimes V$ is
an $\n_\l$-module, where $\n_\l$ is acting on $V$ by restriction of
the $\g$-action and on the tensor product through the
comultiplication in $U(\n_\l)$.  We refer to this action of $\n_\l$
as the {\em dot action}; it is given explicitly by
\begin{equation} \label{e:nact}
x \cdot ((u\1_\l \otimes m) \otimes v) = ([x,u]\1_\l \otimes m)
\otimes v + (u\1_\l \otimes m) \otimes xv.
\end{equation}

\begin{Definition} \label{D:trans}
We define the {\em translation of $M$ by $V$}
$$
M \ostar_\l V = H^0(\n_\l,(Q_\l \otimes_{W_\l} M) \otimes V),
$$
where the cohomology is taken with respect to the dot action of
$\n_\l$, i.e.\ $M \ostar_\l V$ is the space of invariants of $(Q_\l
\otimes_{W_\l} M) \otimes V$ with respect to the dot action of
$\n_\l$.
\end{Definition}

We have that $(Q_\l \otimes_{W_\l} M) \otimes V$ is a $U(\g)$-module
through the comultiplication in $U(\g)$.  It is straightforward to
check that this gives rise to a well-defined $W_\l$-module structure
on $M \ostar_\l V$ with action defined by $u\1_\l z = uz$ for
$u\1_\l \in W_\l$ and $z \in (Q_\l \otimes_{W_\l} M) \otimes V$.

\subsection{Translation of $U(\g,e)$-modules} \label{ss:nllatrans}

In this subsection we define translation for $U(\g,e)$-modules. We
show, in Lemma~\ref{L:transequiv} that this definition agrees with
that for $W$-modules through the isomorphism given in Lemma~\ref{L:equiv}.

First we define a ``comultiplication'' $\tilde \Delta : U(\tilde \g)
\to U(\tilde \g) \otimes U(\g)$ by
$$
\tilde \Delta(x) = x \otimes 1 + 1 \otimes x \quad \text{ and }
\quad \tilde \Delta(y^{\ne})  = y^{\ne} \otimes 1
$$
for $x \in \g$ and $y \in \k$.
It is trivial to check that $\tilde \Delta([x,y]) = \tilde \Delta(x)
\tilde \Delta(y) - \tilde \Delta(y) \tilde \Delta(x)$ for all $x, y
\in \tilde \g$.   Therefore, given a $U(\tilde \g)$-module $E$ and a
$U(\g)$-module $V$, we can give $E \otimes V$ the structure of a
$U(\tilde \g)$, through $\tilde \Delta$.  Moreover, if $V$ is finite
dimensional, then $? \otimes V$ defines an exact endofunctor of
$U(\tilde \g)$-mod. It is clear that if $E \in \Wh(\tilde \g, e)$,
then $E \otimes V \in \Wh(\tilde \g, e)$.  Therefore, we may
transport the functor $? \otimes V$ through the version of
Skryabin's equivalence in Theorem~\ref{T:skryabin} to obtain an
exact endofunctor of $U(\g,e)$-mod, as defined below.

For the rest
of this subsection let $M$ be a $U(\g,e)$-module and $V$ a finite dimensional
$U(\g)$-module.

\begin{Definition} \label{D:nllatrans}
The {\em translation of $M$ by $V$} is defined
to be
$$
M \ostar V = H^0(\n,(\tilde Q \otimes_{U(\g,e)} M) \otimes V),
$$
where the cohomology is taken with respect to the dot action of
$\n$ given in \eqref{e:dot}.
\end{Definition}

To be more explicit we note that the dot action of $\n$ on
$(\tilde Q \otimes_{U(\g,e)} M) \otimes V$ is given by
\begin{equation} \label{e:dotnlla}
x \cdot (u\tilde \1 \otimes m \otimes v) = [x-x^{\ne},u]\tilde \1 \otimes m
\otimes v + u\tilde \1 \otimes m \otimes xv,
\end{equation}
and $M \ostar V$ is the space of invariants for this action.

The next lemma says that our definitions of translation are
equivalent, it follows easily from Lemma~\ref{L:equiv} using
\eqref{e:dotnlla} and \eqref{e:nact}.

\begin{Lemma} \label{L:transequiv}
View $M$ as a $W$-module through the isomorphism
given in Lemma~\ref{L:equiv}.  Then the natural map $U(\tilde \g)
\to Q$ given by $u \mapsto u \cdot \1$ induces an isomorphism
$$
(\tilde Q \otimes_{U(\g,e)} M) \otimes V \isoto (Q \otimes_W M)
\otimes V,
$$
which intertwines the dot actions of $\n$ on $(\tilde Q
\otimes_{U(\g,e)} M)$ and $(Q \otimes_W M) \otimes V$.  Therefore,
we have
$$
H^0(\n,(\tilde Q \otimes_{U(\g,e)} M) \otimes V) \iso H^0(\n,(Q
\otimes_W M) \otimes V).
$$
\end{Lemma}

\section{Vector space isomorphism and lift matrices} \label{S:vspiso}

In this section we show that the translation $M \ostar V$ of a
$U(\g,e)$-module $M$ by a $U(\g)$-module $V$ is isomorphic as a
vector space to $M \otimes V$.  We initially do this for the
Whittaker model definition by considering suitable Kazhdan
filtrations in \S\ref{ss:kazh}.  Then we discuss explicit
isomorphisms using {\em lift matrices} in \S\ref{ss:lift}. Lastly we
consider the relationship between translations and the loop
filtration in \S\ref{ss:loop}.

\subsection{The Kazhdan filtration} \label{ss:kazh}

We use the notation from \S\ref{ss:whittw}.  Let $M$ be a
$W_\l$-module and $V$ a finite dimensional $U(\g)$-module. Below we
consider a Kazhdan filtration of $M \ostar_\l V$ and show that the
associated graded module $\gr (M \ostar_\l V)$ is a $\gr
W_\l$-module isomorphic as a vector space to $\gr M \otimes V$.

Choose a finite dimensional subspace $M_0$ of $M$ that generates $M$
and define $F_j M = (F_j W_\l) M_0$,
where $F_j W_\l$ denotes the $j$-part of the Kazhdan filtration on
$W_\l$.  This defines a (Kazhdan) filtration on $M$, so that $M$ is
a filtered $W_\l$-module; thus $\gr M$ is a $\gr W_\l$-module. The
semisimple element $c \in \g$ (that defines the good grading for
$e$) acts diagonally on $V$, and its eigenspace decomposition gives
a grading and thus also a filtration of $V$. These filtrations on
$M$ and $V$ along with the Kazhdan filtration on $Q_\l$ determine a
filtration of $(Q_\l \otimes_{W_\l} M) \otimes V$.  It is clear that
$(Q_\l \otimes_{W_\l} M) \otimes V$ is a filtered $U(\g)$-module for
the Kazhdan filtration; therefore, $M \ostar_\l V$ is a filtered
$W_\l$-module and the associated graded module $\gr (M \ostar_\l V)$
is a module for $\gr W_\l$. Also we can give $\gr M \otimes V$ the
structure of a $\gr W_\l$-module with trivial action on $V$, i.e.\
$u (m \otimes v) = um \otimes v$ for $u \in \gr W_\l$, $m \in \gr M$
and $v \in V$.

We are now in a position to state and prove the main theorem of this
subsection, which, in particular, implies that there is a vector
space isomorphism $M \ostar_\l V \iso M \otimes V$.

\begin{Theorem} \label{T:iso}
There is an isomorphism of $\gr W_\l$-modules
$$
\gr(M \ostar_\l V) \iso \gr M \otimes V.
$$
\end{Theorem}

\begin{proof}
The good grading of $\g$ gives a grading, and so a filtration, of
$\n_\l$. One can check that $(Q_\l \otimes_{W_\l} M) \otimes V$ is a
filtered $\n_\l$-module for the dot action of $\n_\l$ given by
\eqref{e:nact}, therefore, $\gr((Q_\l \otimes_{W_\l} M) \otimes V)$
is a module for $\gr \n_\l = \n_\l$. We have
\begin{align*}
\gr ((Q_\l \otimes_{W_\l} M) \otimes V) & \iso (\gr Q_\l
\otimes_{\gr W_\l} \gr M) \otimes \gr V \\ & \iso (\C[N_\l] \otimes
\C[\cS]) \otimes_{\C[\cS]} \gr M \otimes V \\ &\iso \C[N_\l] \otimes
\gr M \otimes V.
\end{align*}
In the above, we use the isomorphisms $\gr W_\l \iso \C[\cS]$ and
$\gr Q_\l \iso \C[N] \otimes \C[\cS]$ from \S\ref{ss:whittw}, and
the fact that $\gr V \iso V$, because $V$ is already graded.  We
explain the action of $\n_\l$ on $\C[N_\l] \otimes \gr M \otimes V$.
The action of $N_\l$ on itself by left translations induces a
locally finite representation of $N_\l$ in $\C[N_\l]$ by $(x \cdot
f)(n) = f(x^{-1}n)$.  The action of $\n_\l$ on $V$ can be
exponentiated to give an action of $N_\l$ on $V$.  Then $\C[N_\l]
\otimes V$ is a locally finite $N_\l$-module with the diagonal
action of $N_\l$, and $\C[N_\l] \otimes \gr M \otimes V$ is a
locally finite $N_\l$-module with trivial $N_\l$ action on $\gr M$.
This $N_\l$-module structure differentiates to give the
$\n_\l$-module structure on $\C[N_\l] \otimes \gr M \otimes V$.
Taking $\n_\l$-invariants in $\C[N_\l] \otimes \gr M \otimes V$ is
the same as taking $N_\l$-invariants so we get
\begin{align*} H^0(\n_\l,\gr ((Q_\l \otimes_{W_\l} M) \otimes
V)) &\iso   H^0(\n_\l, \C[N_\l] \otimes \gr M \otimes V) \\
& \iso \gr M \otimes (\C[N_\l] \otimes V)^{N_\l}.
\end{align*}
It is a standard result that $(\C[N_\l] \otimes V)^{N_\l} \iso V$,
see for example \cite[I.3.7(6)]{Ja}; in fact the map $\C[N_\l] \to
\C$ given by evaluation at $1$ gives a map $\C[N_\l] \otimes V \to
V$, which restricts to the above isomorphism.  In turn this gives an
isomorphism
\begin{equation} \label{e:epsilon}
\epsilon : H^0(\n_\l,\gr((Q_\l \otimes_{W_\l} M) \otimes V)) \iso
\C[N_\l] \otimes \gr M \otimes V \isoto \gr M \otimes V.
\end{equation}

Next we want to show that the natural map
\begin{equation} \label{e:gr}
\gr H^0(\n_\l, (Q_\l \otimes_{W_\l} M) \otimes V) \to H^0(\n_\l, \gr
((Q_\l \otimes_{W_\l} M) \otimes V))
\end{equation}
is an isomorphism.  To do this we use the standard spectral sequence
for calculating the cohomology of the filtered module $(Q_\l
\otimes_{W_\l} M) \otimes V$, which we denote by $(E_r)$, i.e.\ the
standard complex for calculating the $\n_\l$-cohomology of $(Q_\l
\otimes_{W_\l} M) \otimes V$ is filtered, and one takes $(E_r)$ to
be the corresponding spectral sequence. It is a standard result that
$H^i(\n_\l, \C[N_\l] \otimes V) = 0$ for $i > 0$: this can be proved
by choosing a filtration $F_j'V$ of $V$ as an $\n_\l$-module, such
that the action of $\n_\l$ on the associated graded module $\gr' V$
is trivial; then one can apply a spectral sequence argument along
with the fact that $H^i(\n_\l,\C[N_\l]) = 0$ for $i > 0$, as this is
de Rham cohomology of the affine space $N_\l$, see for example
\cite[\S 4.3]{GG}. It follows that $H^i(\n_\l, \gr ((Q_\l
\otimes_{W_\l} M) \otimes V)) = 0$ for all $i
> 0$.  This implies that $E_1$ is concentrated in degree $0$, so the
spectral sequence $(E_r)$ stabilizes at $r = 2$, namely $E_2 =
E_\infty$. Therefore, the map in \eqref{e:gr} is an isomorphism, and
in turn we obtain an isomorphism
\begin{equation} \label{e:barepsilon}
\bar \epsilon : \gr (M \ostar_\l V) \isoto \gr M \otimes V.
\end{equation}
from $\epsilon$ in \eqref{e:gr}.

We are left to show that $\bar \epsilon$ is an isomorphism of $\gr
W_\l$-modules.  First we consider $V$ as a filtered $U(\g)$-module
for the Kazhdan filtration.  It is clear that the action of $\gr
U(\g) \iso S(\g)$ on the associated graded module $\gr V \iso V$ is
trivial.  Therefore the action of $\gr U(\g)$ on $\gr ((Q_\l
\otimes_{W_\l} M) \otimes V))$ is given by
$$
a(u \otimes m \otimes v) = au \otimes m \otimes v
$$
for $a \in \gr U(\g)$, $u \in \gr \tilde Q$, $m \in \gr M$ and $v
\in V$.  Using the fact that $\gr U(\g)$ and $\gr W_\l$ are
commutative. we see that the action of $\gr W_\l$ on $H^0(\n_\l,(\gr
Q_\l \otimes_{\gr W_\l} \gr M) \otimes V)$ is determined by
restricting the formula
$$
a(u \otimes \m \otimes v) = u  \otimes am \otimes v,
$$
for $a \in \gr W_\l$, $u \in \gr \tilde Q$, $m \in \gr M$ and $v \in
V$. From these expressions it is straightforward to see that $\bar
\epsilon$ does indeed restrict to the required isomorphism of $\gr
W_\l$-modules.
\end{proof}

We now translate Theorem~\ref{T:iso} in to setting of translations
of $U(\g,e)$-modules through Lemma~\ref{L:transequiv}. The
discussion below is analogous to that in \cite[\S 3.2]{BGK}.

First we explain the commutative diagram below, which regards the
translation $M \ostar V$ for the case where $M = U(\g,e)$ is the
regular module.
\begin{equation} \label{e:ickdiag}
\xymatrix{
(S(\tilde \p) \otimes V)^\n \ar@{^{(}->}[r] \ar[dr]_-{\sim} &
S(\tilde\p) \otimes V \ar[r]^-{\sim} \ar@{>>}[d]^{\zeta_V} &
\C[e+\m^\perp] \otimes V \ar@{>>}[d]^{\res} &
(\C[e+\m^\perp] \otimes V)^N \ar@{_{(}->}[l] \ar[dl]^-{\sim} \\
& S(\g^e) \otimes V \ar[r]^-{\sim} & \C[\cS] \otimes V. \\ }
\end{equation}
We identify the $U(\tilde \g)$-module $(\tilde Q \otimes_{U(\g,e)}
U(\g,e)) \otimes V$ with $U(\tilde \p) \otimes V$, and $U(\g,e)
\ostar V$ with the subspace of $\n$-invariants for the dot action.
We have a Kazhdan filtration on $U(\tilde \p) \otimes V$, where as
before $V$ is graded, and so filtered, by the $c$-eigenspace
decomposition. The associated graded module is isomorphic to
$S(\tilde \p) \otimes V$. The dot action from \eqref{e:dot} is
filtered and so gives an action of $\n$ on $S(\tilde \p) \otimes V$.
The map of the left of the diagram is the inclusion of the
invariants for this action $(S(\tilde \p) \otimes V)^\n =
H^0(\n,S(\tilde \p) \otimes V)$.

From Lemma~\ref{L:equiv}, we obtain an isomorphism $U(\tilde \p)
\otimes V \isoto Q \otimes V$, which gives the isomorphism
$S(\tilde\p) \otimes V \isoto \C[e+\m^\perp] \otimes V$ in the
diagram through the identification $\gr Q \iso \C[e+\m^\perp]$. This
isomorphism can also be described as follows: we can view
$\C[e+\m^\perp] \otimes V$ as the space of regular functions from $e
+ \m^\perp$ to $V$; then $x \otimes v \in \p \otimes V$ is sent to
the function $z \mapsto (x|z)v$ and $y^\ne \otimes v \in \k^\ne
\otimes V$ to the function $z \mapsto (y|z)v$.

The inclusion of the $N$-invariants (equivalently $\n$-invariants)
in $\C[e+\m^\perp] \otimes V$ is on the right of the diagram.  From
the proof of Theorem~\ref{T:iso}, it follows that there is an
isomorphism $\gr H^0(\n, Q \otimes V) \iso (\C[e+\m^\perp] \otimes
V)^N$. Thus through Lemma~\ref{L:transequiv} we obtain an
isomorphism $\gr (U(\g,e) \ostar V) \iso (S(\tilde \p) \otimes
V)^\n$.

We now consider the vertical maps.  The map $\zeta_V = \zeta \otimes
\id_V$ involves the map $\zeta : S(\tilde \p) \to S(\g^e)$ from
\eqref{e:zeta}. The other vertical map is the restriction map. This
square commutes, because $(x|z) = 0$ for any $x \in \r$
and $z \in e + \g^f$, see the discussion in \cite[\S 3.2]{BGK}.

The diagonal map on the right is an isomorphism, because the
restriction map identifies with the map $\epsilon : \C[N] \otimes
\C[\cS] \otimes V \to \C[\cS] \otimes V$ from \eqref{e:epsilon}.
Hence, also the diagonal map on the left is an isomorphism.

The upshot of this commutative diagram is that the restriction of
$\zeta \otimes \id_V$ gives rise to an isomorphism $\eta_V$ recorded
in the following proposition.  This isomorphism is obtained as the
composition of the diagonal isomorphism in \eqref{e:ickdiag} with
$\zeta^{-1} \otimes \id_V$, where $\zeta^{-1}$ means the inverse of
$\zeta : S(\p)^\n \to S(\g^e)$. The formula given in the proposition
is a consequence of Lemma~\ref{L:zetaeta}, in the statement we use
$\eta$ from \eqref{e:eta}.

\begin{Proposition} \label{P:etaV}
The map $\eta_V : S(\tilde \p) \otimes V \to S(\tilde \p)^\n \otimes
V$, defined by $\eta_V(u \otimes v) = \eta(u) \otimes v$ for $u \in
S(\tilde \p)$ and $v \in V$, restricts to an isomorphism of vector
spaces
\begin{equation} \label{e:etaV}
\eta_V : \gr (U(\g,e) \ostar V) \isoto \gr U(\g,e) \otimes V.
\end{equation}
\end{Proposition}

We now consider the situation for any $U(\g,e)$-module $M$.  There
is an analogue of the commutative diagram \eqref{e:ickdiag}, where
the triangle on the left is
\begin{equation} \label{e:ickdiag2}
\xymatrix{ ((S(\tilde \p) \otimes_{S(\tilde \p)^\n} \gr M ) \otimes
V)^\n \ar@{^{(}->}[r] \ar[dr]_-{\sim} & (S(\tilde \p)
\otimes_{S(\tilde \p)^\n} \gr M ) \otimes V
\ar@{>>}[d]^{\zeta_{M,V}} &
\\
& (S(\g^e) \otimes_{S(\g^e)} \gr M ) \otimes V. \\ }
\end{equation}
Identifying $\gr \tilde Q$ with $S(\tilde \p)$ we have
$$
\gr (M \ostar V) \iso ((S(\tilde \p) \otimes_{S(\tilde \p)^\n} \gr
M) \otimes V)^\n
$$
is the module in the top left of the diagram.  The horizontal map is
simply the inclusion of invariants.  We have that $\gr M$ is a
module for $\gr U(\g,e) \iso S(\p)^\n$, and we may also view $\gr M$
as an $S(\g^e)$-module through $\zeta$ from \eqref{e:zeta}.  Thus,
we can make sense of the module at the bottom of the diagram. The
vertical map $\zeta_{M,V}$ is defined by
$$
\zeta_{M,V}(u \otimes m \otimes v) = \zeta(u) \otimes m \otimes v.
$$
for $u \in S(\tilde \p)$, $M \in \gr M$ and $v \in V$.  Analogous
arguments to the case $M = U(\g,e)$ show that the diagonal map is an
isomorphism.

There is the obvious isomorphism
\begin{equation} \label{e:obvious}
(S(\g^e) \otimes_{S(\g^e)} \gr M ) \otimes V \isoto \gr M \otimes V.
\end{equation}
Thanks to Lemma~\ref{L:zetaeta}, this is the same as the isomorphism
given by $u \otimes m \otimes v \mapsto \zeta^{-1}(u)m \otimes v$
for $u \in S(\g^e)$, $M \in \gr M$ and $v \in V$, where on the
right-hand side we are viewing $\gr M$ as a $S(\tilde
\p)^\n$-module. In turn, this means that the composition of the
diagonal isomorphism in \eqref{e:ickdiag2} with the isomorphism in
\eqref{e:obvious} is given by restricting the map
$$
u \otimes m \otimes v \mapsto \eta(u)m \otimes v
$$
to $\gr(M \ostar V)$, for $u \in S(\tilde \p)$, $m \in \gr M$ and $v
\in V$, where $\eta$ is defined in \eqref{e:eta}.

The above discussion is summarized in the following theorem.
\begin{Theorem} \label{T:etaMV}
The map  $\eta_{M,V} : (S(\tilde \p) \otimes_{S(\tilde \p)^\n} \gr
M) \otimes V \to \gr M \otimes V$, defined by
$\eta_{M,V}(u \otimes m \otimes v) = \eta(u)m \otimes v$
for $u \in S(\tilde \p)$, $m \in \gr M$ and $v \in V$, restricts to
an isomorphism
\begin{equation} \label{e:etaMV}
\eta_{M,V} : \gr(M \ostar V) \isoto M \otimes V.
\end{equation}
\end{Theorem}

\subsection{Explicit isomorphisms and lift matrices} \label{ss:lift}

Throughout this subsection, $M$ is a $U(\g,e)$-module and $V$ is a
finite dimensional $U(\g)$-module. From Theorem~\ref{T:etaMV}, it
follows that there is an isomorphism of vector spaces $M \ostar V
\iso M \otimes V$.  However, there is not in general a canonical
isomorphism, in this subsection we discuss explicit isomorphisms and
describe all isomorphisms satisfying a certain technical condition
explained in Remark~\ref{R:allisos}. Our approach is based on
\cite[Thm.\ 8.1]{BK2}, which in turn is attributed as a
reformulation of \cite[Thm.\ 4.2]{Ly}. First we need to introduce
some notation.

We choose an ordered basis $\bv = (v_1,\dots,v_n)$ of $V$,
consisting of $\t$-weight vectors.  We let $c_i$ be the eigenvalue
of $c$ on $v_i$, and assume that $\bv$ is ordered so that $c_1 \ge
\dots \ge c_n$. The coefficient functions $b_{ij} \in U(\g)^*$ of
$V$ are defined by $uv_j = \sum_{i=1}^n b_{ij}(u) v_i$.  We say that
an $n \times n$-matrix $\bz = (z_{ij})$ is {\em $V$-block lower
unitriangular} if $z_{ij} = \delta_{ij}$ whenever $c_i \ge c_j$.

Let $(x_1^{\ne}),\dots,(x_{2s}^{\ne}), x_{2s+1},\dots,x_r$ be the
basis of $\r = \bigoplus_{j \ge 2} [f,\g(j)] \oplus \k^\ne$
introduced after Lemma~\ref{L:free}. The right $U(\g,e)$-module
homomorphism $\chi : \tilde Q \to U(\g,e)$ is defined in
\eqref{e:chi}. We define $\chi_{M,V} :(\tilde Q \otimes_{U(\g,e)} M)
\otimes V \to M \otimes V$ by
\begin{equation} \label{e:chiMV}
\chi_{M,V}(u\tilde \1 \otimes m \otimes v) = \chi(u\tilde \1)m
\otimes v.
\end{equation}
The following theorem gives the desired vector space isomorphism
between $M \ostar V$ and $M \otimes V$.

\begin{Theorem} \label{T:projiso}
The restriction of $\chi_{M,V}$ to $M \ostar V$ is an isomorphism of
vector spaces
$$
\chi_{M,V} : M \ostar V \to M \otimes V.
$$
Moreover, $\chi$ is natural in both $M$ and $V$.
\end{Theorem}

\begin{proof}
First we consider the case $M = U(\g,e)$ is the regular module.  In
this case we identify $(\tilde Q \otimes _{U(\g,e)} U(\g,e)) \otimes
V$ with $\tilde Q \otimes V$, and write $\chi_V = \chi_{M,V}$. Under
this identification we have $\chi_V(u \tilde \1 \otimes v) = \chi(u
\tilde \1) \otimes v$ for $u \in U(\tilde \p)$ and $v \in V$.  Thus
the associated graded map $\gr \chi_V : S(\tilde \p) \otimes V \to
\gr (U(\g,e) \otimes V)$ is the same as $\eta_V$ from
\eqref{e:etaV}. Thus by Proposition~\ref{P:etaV}, we have that $\gr
\chi_V$ is an isomorphism.  A standard filtration argument now tells
us that $\chi_V$ is an isomorphism.

The case for general $M$ is similar: we follow the same arguments
identifying $\gr \chi_{M,V}$ with $\eta_{M,V}$ from \eqref{e:etaMV}
then appealing to Theorem~\ref{T:etaMV}.

It is clear that $\chi$ is natural in both $M$ and $V$.
\end{proof}

The following lemma is similar to parts \cite[Thm.\ 8.1]{BK2},
though the proof here is different. We recall that $\Pr : U(\tilde
\g) \to U(\tilde \p)$ is the projection along the direct sum
decomposition $U(\tilde \g) = U(\tilde \p) \oplus \tilde I$.

\begin{Lemma} \label{L:inverse}
Let $\bx^0 = (x^0_{ij})$ be an $n \times n$ matrix with entries in
$U(\tilde \p)$ satisfying the conditions:
\begin{enumerate}
\item[(i)] $\chi(x^0_{ij}) = \delta_{ij}$, and
\item[(ii)]
$\Pr([x-x^{\ne},x^0_{ij}]) + \sum_{k=1}^n b_{ik}(x) x^0_{kj} \in
\tilde I$, for all $x \in \n$.
\end{enumerate}
Then the inverse to $\chi_{M,V} : M \ostar V \to M \otimes V$ is
given by the map $\psi_{\bx^0,M,\bv} : M \otimes V \to M \ostar V$
defined by $\psi_{\bx^0,M,\bv}(m \otimes v_j) = \sum_{i=1} ^n
x^0_{ij} \otimes m \tilde \1 \otimes v_i$, for $m \in M$.

Moreover, $\bx^0$ is uniquely determined by conditions (i) and (ii)
and is $V$-block lower unitriangular.
\end{Lemma}

\begin{proof}
We first consider the case $M = U(\g,e)$ and identify $(\tilde Q
\otimes_{U(\g,e)} U(\g,e)) \otimes V$ with $\tilde Q \otimes V$. We
write $\psi_{\bx^0,\bv}$ for $\psi_{\bx^0,M,\bv}$ in this case.

It is clear that the inverse of $\chi_V$ must have the form
$\psi_{\bx^0,\bv}$ for some matrix $\bx^0$ with entries in $U(\tilde
\p)$.  In particular, $\psi_{\bx^0,V}(1 \otimes v_j) = \sum_{i=1}^N
x^0_{ij} \tilde \1 \otimes v_i$.  This forces $\chi(x_{ij}) =
\delta_{ij}$ giving (i). Also $\sum_{i=1}^N x^0_{ij}\tilde \1
\otimes v_i$ must lie in $M \ostar V = H^0(\n,(\tilde Q
\otimes_{U(\g,e)} U(\g,e)) \otimes V)$, which forces (ii) to hold.

Conversely, one can check that conditions (i) and (ii) imply that
$\psi_{\bx^0,\bv}$ is inverse to $\chi_V$, so that they uniquely
determine $\bx^0$.  To see that $\bx^0$ is $V$-block lower
unitriangular, one observes that conditions (i) and (ii) would still
hold if we replaced $x_{ij}^0$ with $\delta_{ij}$ when $c_i \ge
c_j$.

For general $M$, it is immediate from (i) and (ii) that $\chi_{M,V}
\psi_{\bx^0,M,\bv} = \id_{M \otimes V}$.  Therefore,
$\psi_{\bx^0,M,\bv}$ is inverse to $\chi_{M,V}$, and the proof is
complete.
\end{proof}

The matrix $\bx^0$ in Lemma~\ref{L:inverse} leads us to the
following definition of lift matrices, which is key to a number of
results in the sequel.

\begin{Definition} \label{D:lift}
Let $\bx = (x_{ij})$ be a $V$-block lower unitriangular matrix with
entries in $U(\tilde \p)$ satisfying:
\begin{equation} \label{e:lift}
\Pr([x-x^{\ne},x_{ij}]) + \sum_{k=1}^n b_{ik}(x) x_{kj} = 0
\end{equation}
for all $x \in \n$, and $x_{ij} \in F_{c_j - c_i} U(\tilde \p)$.
Then we call $\bx$ a {\em lift matrix} for the basis $\bv$ of $V$.
\end{Definition}

We note in particular that $\bx^0$ is a lift matrix for $\bv$, the
condition on the filtered degree of entries to $\bx^0$ holds because
$\chi_{M,V}$ is clearly a filtered map. Our next proposition shows
that all lift matrices give rise to a vector space isomorphism $M
\otimes V \iso M \ostar V$. It is proved by adapting arguments from
the proof of \cite[Thm.\ 8.1]{BK2}.

\begin{Proposition} \label{P:lift}
Let $\bx^0$ be as in Lemma~\ref{L:inverse}.
\begin{enumerate}
\item[(a)]
Let $\bx$ be a lift matrix for $\bv$.
\begin{enumerate}
\item[(i)]
The map $\psi_{\bx,M,\bv} =  : M \otimes V \to M \ostar V$ defined
by $\psi_{\bx,M,\bv}(m \otimes v_j) = \sum_{i=1}^n x_{ij} \tilde \1
\otimes m \otimes v_i$ is a vector space isomorphism.
\item[(ii)]
We have $\bx = \bx^0 \bw^0$, where $\bw^0$ is a $V$-block lower
unitriangular matrix with entries in $U(\g,e)$
\end{enumerate}
\item[(b)]  Suppose $\bx$ is of the form  $\bx = \bx^0 \bw^0$, where
$\bw^0 = (w^0_{ij})$ is a $V$-block lower unitriangular matrix with
$w^0_{ij} \in F_{c_j -c_i} U(\g,e)$.  Then $\bx$ is a lift matrix
for $\bv$.
\item[(c)] Let $\bx$ be a lift matrix and $\by$ an $\n times n$ matrix with entries in $U(|tilde \p)$.
Then $\by$ is a lift matrix if and only if there is a $V$-block
lower unitriangular matrix $\bw= (w_{ij})$ with $w_{ij} \in F_{c_j
-c_i} U(\g,e)$, such that $\bx = \by \bw$.
\end{enumerate}
\end{Proposition}

\begin{proof}
Let $\bx$ be an lift matrix for $\bv$. Consider the case $M =
U(\g,e)$, and as usual identify $(\tilde Q \otimes_{U(\g,e)}
U(\g,e)) \otimes V$ with $\tilde Q \otimes V$. Since $\bx$ satisfies
\eqref{e:lift}, we have that $\sum_{i=1}^n x_{ij} \tilde \1 \otimes
v_i \in U(\g,e) \ostar V$.  It follows from Lemma~\ref{L:inverse}
that there exist $w^0_{kj} \in U(\g,e)$ such that
$$
\sum_{i=1}^n x_{ij} \otimes v_i = \sum_{i,k=1}^n x^0_{i,k} w^0_{k,j}
\otimes v_i
$$
Equating coefficients gives $\bx = \bx^0 \bw^0$, where $\bw^0 =
(w^0_{ij})$. Since $\bx$ and $\bx^0$ are $V$-block lower
unitriangular, so is $\bw^0$.  This proves (a)(ii).

Now consider general $M$.  From the factorization $\bx = \bx^0
\bw^0$, we see that $\psi_{\bx,M,\bv}$ is the composition of the map
$M \otimes V \to M \otimes V$ given by
\begin{equation} \label{e:wiso}
m \otimes v_j \mapsto \sum_{i=1}^n w^0_{ij} m \otimes v_i
\end{equation}
with $\psi_{\bx^0,M,\bv}$. The map in \eqref{e:wiso} is an
isomorphism as $\bw$ is lower unitriangular and thus invertible, and
$\psi_{\bx^0,M,\bv}$ is an isomorphism by Lemma~\ref{L:inverse}.
Therefore, $\psi_{\bx,M,\bv}$ is an isomorphism proving (a)(i).

One can check via a straightforward calculation that if $\bx$ is of
the form given in (b), then it satisfies the conditions to be a lift
matrix in Definition~\ref{D:lift}.  Then (c) is a consequence of
(a)(ii) and  (b); it is straightforward to check that the condition
on filtered degrees holds.
\end{proof}

\begin{Remark} \label{R:allisos}
We now describe all the isomorphisms of vector spaces given by
Proposition~\ref{P:lift}(a)(i).  This is only really possible in
case $M = U(\g,e)$, so we restrict to this situation.  We identify
$(\tilde Q \otimes_{U(\g,e)} U(\g,e)) \otimes V$ with $\tilde Q
\otimes V$.

Consider an isomorphism $\psi : U(\g,e) \otimes V \to U(\g,e) \ostar
V$ of vector spaces that is filtered with respect to the Kazhdan
filtration. We have $\psi(\tilde \1 \otimes v_j) = \sum_{c_i = c_j}
a_{ij} \tilde \1 \otimes v_j + \sum_{c_i < c_j} z_{ij}\tilde \1$,
where $a_{ij} \in \C$ and $z_{ij} \in U(\tilde \p)$.  The
isomorphisms in Proposition~\ref{P:lift} are precisely those for
which $a_{ij} = \delta_{ij}$. Thus any such filtered isomorphism
$\psi$ can be factorized as the composition an isomorphism $U(\g,e)
\otimes V \isoto U(\g,e) \otimes V$ of the form $u \otimes v_j
\mapsto \sum_{c_i = c_j} a_{ij} u \otimes v_i$, with an isomorphism
$\psi_{\bx,\bv}$ from Proposition~\ref{P:lift}.

Another characterization of the isomorphisms from Proposition
\ref{P:lift}(a)(i) is given in term of the loop filtration discussed
in the next subsection.  Note that $\psi: U(\g,e) \ostar V \isoto
U(\g,e) \otimes V$ being filtered with respect to the Kazhdan
filtration implies that it is filtered with respect to the loop
filtration.  Then the $\psi_{\bx,\bv}$ are precisely those $\psi$
such that the associated graded map with respect to the loop
filtration identifies with the identity map through Proposition
\ref{P:loop}.
\end{Remark}

\subsection{The loop filtration} \label{ss:loop}

In Proposition~\ref{P:loop} below, we prove a compatibility result
regarding the loop filtration and translation.  Throughout this
subsection let $M$ be a $U(\g,e)$-module generated by the finite
dimensional subspace $M_0$, and let $V$ be a finite dimensional
$U(\g)$-module with ordered basis $\bv$ as in \S\ref{ss:lift}.

Using the loop filtration $(F_j' U(\g,e))_{j \in \Z_{\ge 0}}$ of
$U(\g,e)$ from \S\ref{ss:nllaw} we can define a loop filtration on
$M$ by setting $F_j' M = (F_j' U(\g,e)) M_0$. In this way $M$
becomes a filtered $U(\g,e)$-module, so the associated graded module
$\gr' M$ is a module for $\gr' U(\g,e) \iso U(\g^e)$.  Also $V$ is a
$U(\g^e)$-module by restriction, so $(\gr' M) \otimes V$ has the
structure of a $U(\g^e)$-module. There is a loop filtration on
$\tilde Q$ through the identification of $\tilde Q \iso U(\tilde
\p)$.  As before the $c$-eigenspace decomposition gives a grading
and thus also a filtration of $V$. Putting this all together we
obtain loop filtrations of $(\tilde Q \otimes_{U(\g,e)} M) \otimes
V$ and $M \ostar V$.  It is easy to check that $M \ostar V$ is a
filtered $U(\g,e)$-module; therefore, the associated graded module
$\gr' (M \ostar V)$ is a module for $\gr' U(\g,e) \iso U(\g^e)$.

We can now state and prove the compatibility result of this
subsection.

\begin{Proposition} \label{P:loop}
There is an isomorphism $\gr' (M \ostar V) \iso (\gr' M) \otimes V$
of $U(\g^e)$-modules.
\end{Proposition}

\begin{proof}
We just consider the case $M = U(\g,e)$, the general case can be
dealt with similarly, and we leave the details to the reader.  As
usual we identify $(\tilde Q \otimes_{U(\g,e)} U(\g,e)) \otimes V$
with $\tilde Q \otimes V$.

Let $\bx$ be a lift matrix for $\bv$, and let $\psi_{\bx,\bv} :
U(\g,e) \otimes V \isoto U(\g,e) \ostar V$ be the corresponding
isomorphism. Consider $\psi_{\bx,\bv}(1 \otimes v_j) = \sum_{i=1}^n
x_{ij}\tilde \1 \otimes v_i$. The condition on the filtered degree
of entries $\bx$ in the Kazhdan filtration in Definition
\ref{D:lift} means that each of the terms $x_{ij}\tilde \1 \otimes
v_i$ for $i \neq j$ is zero or has strictly lower degree than
$\tilde \1 \otimes v_j$ in the loop filtration. Recalling the
isomorphism $\theta$ from \eqref{e:theta}, we deduce that $\theta
\otimes \id_V : U(\g^e) \otimes V \to U(\p) \ostar V$ maps
isomorphically onto $\gr'(U(\g,e) \ostar V)$. It is clear that this
isomorphism respects the $U(\g^e)$-module structure, so the proof is
complete.
\end{proof}

As a corollary we obtain the following result, which can be proved
using a standard PBW basis argument.

\begin{Corollary} \label{C:gen}
Suppose $\gr' M$ is generated by $\gr' M_0$ as a $U(\g^e)$-module
and let $\bx$ be a lift matrix for $\bv$. Then $M \ostar V$ is
generated by $\psi_{\bx,M,\bv}(M_0 \ostar V)$.
\end{Corollary}

\section{Basic properties of translation} \label{S:basic}

In this section we record some basic properties of translations;
they are generalizations of results from \cite[\S8.2]{BK2}.

\subsection{Tensor identity}
\label{ss:tildepmod}

Proposition~\ref{P:tildepmod} below is a generalization of
\cite[Cor.\ 8.2]{BK2} from the case where $\g = \gl_n(\C)$
(when there is an even good
grading for $e$); it can be proved using the arguments in {\em loc.\
cit.\ }so we do not include a proof here. For the statement, note
that we can view a $U(\tilde \p)$-module $M$ as a $U(\g,e)$-module
by restriction. Therefore, $M \ostar V$ is defined as a
$U(\g,e)$-module for a finite dimensional $U(\g)$-module $V$.  Also
$M \otimes V$ can be viewed as a $U(\tilde \p)$-module through
$\tilde \Delta$, and thus as a $U(\g,e)$-module by restriction.

\begin{Proposition} \label{P:tildepmod}
Let $M$ be a $U(\tilde \p)$-module and $V$ a finite dimensional
$U(\g)$-module.  Then
\begin{enumerate}
\item[(i)] The restriction of the map $(\tilde Q \otimes_{U(\g,e)} M)
\otimes V \to M \otimes V$ defined by
$$
u\tilde \1 \otimes m \otimes v \mapsto um \otimes v,
$$
for $u \in U(\tilde \p)$, $m \in M$ and $v \in V$, defines a
canonical natural isomorphism
$$
\mu_{M,V} : M \ostar V \isoto M \otimes V.
$$
\item[(ii)] Let $\bv = (v_1,\dots,v_n)$ be a basis of $V$ as in
\S\ref{ss:lift}, let $\bx$ be a lift matrix for $\bv$ and $\by$ the
inverse to $\bx$. Then the inverse map to $\mu_{M,V}$ sends $m
\otimes v_j$ to $\sum_{i,k=1}^n x_{ik} \tilde \1 \otimes y_{kj} m
\otimes v_i$, for $m \in M$.
\end{enumerate}
\end{Proposition}

\subsection{Associativity and adjunction}
\label{ss:assocdual}

In the statement of Lemma~\ref{L:assoc}, we use the isomorphisms of
the form $\chi_{M,V}$ from \eqref{e:chiMV}.  The lemma can be proved
in exactly the same way as in \cite[\S 8.2]{BK2}; so we omit the
details.

\begin{Lemma} \label{L:assoc}
Let $M$ be a $U(\g,e)$-module and $V$, $V'$ be finite-dimensional
$U(\g)$-modules.  Then the linear map from $(\tilde Q
\otimes_{U(\g,e)} ((\tilde Q \otimes_{U(\g,e)} M) \otimes V))
\otimes V' \to (\tilde Q \otimes_{U(\g,e)} M) \otimes (V \otimes
V')$ given by
$$
(u'\tilde \1 \otimes ((u \tilde \1 \otimes m) \otimes v)) \otimes v'
\to (u'(u \tilde \1 \otimes m) \otimes v) \otimes v'),
$$
for $u', u \in U(\tilde \p)$, $m \in M$ $v \in V$ and $v' \in V'$,
restricts to a natural isomorphism
$$
a_{M,V,V'} : (M \ostar V) \ostar V' \isoto M \ostar(V \otimes V').
$$
Moreover the following diagram commutes
$$
\begin{CD}
(M \ostar V) \ostar V' & @>{a_{M,V,V'}}>> & M \ostar (V \otimes V') \\
@V \chi_{M\ostar V,V'} VV&& @VV \chi_{M,V \otimes V'} V\\
(M \ostar V) \otimes V' &@>>{\chi_{M,V} \otimes \id_{V'}}>  & M
\otimes V \otimes V'
\end{CD}
$$
\end{Lemma}

Next we transport the canonical adjunction as in \cite[\S8.2]{BK2}.
In order to do this, note that for the trivial $U(\g)$-module $\C$,
and a $U(\g,e)$-module $M$ there is a natural map $M \ostar \C \to
M$ determined by $i_M(\tilde \1 \otimes m \otimes 1) = m$, for $m
\in M$. Let $v_1,\dots,v_n$ be a basis of $V$, and let
$v^1,\dots,v^n$ be the dual basis of $V^*$. Then the unit of the
canonical adjunction is given by the composition
$$
\begin{CD}
M  &@>{i_M^{-1}}>> & M \ostar \C & @>>> & M \ostar (V \otimes V^*) &
@>{a_{M,V,V^*}^{-1}}>> & (M \ostar V) \ostar V^*,
\end{CD}
$$
where the second map is given by $(\tilde \1 \otimes m) \otimes 1
\to (\tilde \1 \otimes m) \otimes (\sum_{i=1}^n v_i \otimes v^i)$,
for $m \in M$, and the third map is from Lemma~\ref{L:assoc}.  The
counit is given by the composition
$$
\begin{CD}
(M \ostar V^*) \ostar V  &@>{a_{M,V^*,V}}>> &M \ostar (V^* \otimes
V) & @>>> & M \ostar \C & @>{i_M}>> & M,
\end{CD}
$$
where the second map is the restriction of $(u \tilde \1 \otimes m)
\otimes (f \otimes v) =  (u \tilde \1 \otimes m) \otimes \lan
f,v\ran$, for $u \in U(\tilde \p)$, $m \in M$, $f \in V^*$ and $v
\in V$, where $\lan f,v\ran$ denotes the natural pairing.

Summarizing the above discussion we obtain.

\begin{Lemma} \label{L:adjdual}
Let $V$ be a finite dimensional $U(\g)$-module and $V^*$ the dual
$U(\g)$-module.  Then $? \ostar V$ and $? \ostar V^*$ are biadjoint
functors.
\end{Lemma}

\section{Translations and highest weight theory} \label{S:hw}

In this section we consider the relationship between translations
and the highest weight theory from \cite[\S4]{BGK}. First in
\S\ref{ss:hwrec}, we give a brief recollection of some definitions
from highest weight theory. Then in \S\ref{ss:catO}, we recall the
definition of the category $\cO(e)$ of $U(\g,e)$-modules from
\cite[\S4.4]{BGK}, and show that it is stable under translations;
the category $\cO(e)$ is an analogue of the usual BGG category $\cO$
of $U(\g)$-modules. We consider translations of Verma modules for
$U(\g,e)$ in \S\ref{ss:verma}, and in particular show that the
translation of a Verma module has a filtration by Verma modules, ,
in Theorem \ref{T:filtverma}.

\subsection{Recollection on highest weight theory} \label{ss:hwrec}

The recollection below is as brief as possible for our purposes.
Full details can be found in \cite[\S 4]{BGK}.

The {\em restricted root system} $\Phi^e \sub (\t^e)^*$ associated
to $e$ is defined by the $\t^e$-weight space decomposition
$$
\g = \g_0 \oplus \bigoplus_{\alpha \in \Phi^e} \g_\alpha,
$$
where $\g_\alpha = \{x \in \g \mid [t,x] = \alpha(t)x \text{ for all
} t \in \t^e\}$, i.e.\ $\Phi^e$ consists of the elements of $\Phi$
restricted to $\t^e$.  The reader is referred to \cite[\S2 and
\S3]{BG} for information on restricted root systems.

We have $e \in \g_0$ and the good grading on $\g$ gives a good
grading of $\g_0$.  Thus the finite $W$-algebra $U(\g_0,e)$ is
defined in analogy to $U(\g,e)$.  The good grading of $\g_0$ for $e$
must be even as $e$ is distinguished in $\g_0$, see for example
\cite[5.7.6]{Ca}. Therefore $U(\g_0,e) \sub U(\p_0)$, where $\p_0 =
\p \cap \g_0$.  The analogue $\theta_0 : U(\g_0^e) \to U(\p_0)$ of
$\theta$ from \eqref{e:theta0} is simply the inclusion, so we can
view $U(\g_0^e(0))$ as a subalgebra of $U(\g_0,e)$.

From $\theta$ we obtain an embedding $\t^e \into U(\g,e)$, which we
use to identify $\t^e$ with a subalgebra of $U(\g,e)$. Thus there is
an adjoint action of $\t^e$ on $U(\g,e)$ giving the restricted root
space decomposition
$$
U(\g,e) = \bigoplus_{\alpha \in \Z\Phi^e} U(\g,e)_\alpha.
$$
of $U(\g,e)$, where $U(\g,e)_\alpha = \{ u \in U(\g,e) \mid [t,u] =
\alpha(t)u \text{ for all } t \in \t^e\}$. We choose a system
$\Phi^e_+$ of positive roots in the restricted root system $\Phi^e$;
we recall from \cite[\S2]{BG} that this is equivalent to choosing a
parabolic subalgebra $\q$ of $\g$ with Levi subalgebra $\g_0$.  This
choice of positive roots gives rise to a partial order on $(\t^e)^*$
in the usual way, i.e.\ $\alpha \le \beta$ if and only if $\beta -
\alpha \in \Z_{\ge 0} \Phi^e_+$.

We define $U(\g,e)_\sharp$ to be the left ideal of $U(\g,e)$
generated by $U(\g,e)_\alpha$ for $\alpha \in \Phi^e_+$.  Then by
\cite[Thm.\ 4.3]{BGK}, $U(\g,e)_{0,\sharp} = U(\g,e)_0 \cap
U(\g,e)_\sharp$ is a two-sided ideal of $U(\g,e)_0$, and  the
quotient $U(\g,e)_0/U(\g,e)_{0,\sharp}$ is isomorphic to
$U(\g_0,e)$. Next we explain this isomorphism explicitly.  We define
$U(\tilde \p)_0$ and $U(\tilde \p)_{0,\sharp}$ in analogy to
$U(\g,e)_0$ and $U(\g,e)_{0,\sharp}$. We have $U(\tilde \p)_0 =
U(\p_0) \oplus U(\tilde \p)_{0,\sharp}$.
Thus we may
define the projection $\pi : U(\tilde \p)_0 \onto U(\p_0)$ along
this decomposition.  Recall that $b_1,\dots,b_r$ is a basis for $\n$
with $b_i \in \g(-d_i)$ of weight $\beta_i \in \Phi$. By \cite[Lem.\
4.1]{BGK}
$$
\gamma = \sum_{\substack{1 \leq i \leq r \\
\beta_i|_{\t^e}\in \Phi^e_-}}\beta_i
$$
is a character of $\p_0$, so we can define the shift $S_{-\gamma} :
U(\p_0) \to U(\p_0)$ by $S_{-\gamma}(x) =  x - \gamma(x)$.  Now
\cite[Thm.\ 4.3]{BGK} says that composition
\begin{equation} \label{e:hwiso}
\xymatrix{U(\g,e)_0 \ar@{^{(}->}[r] & U(\tilde \p)_0
\ar@{>>}[r]^{\pi} & U(\p_0) \ar[r]^{S_\gamma} & U(\p_0)} \\
\end{equation}
has image equal to $U(\g_0,e)$ and kernel equal to
$U(\g,e)_{0,\sharp}$; giving the desired isomorphism
$U(\g,e)_0/U(\g,e)_{0,\sharp} \isoto U(\g_0,e)$.

Given a finite dimensional $U(\g_0,e)$-module $L$ we define the
induced $U(\g,e)$-module
$$
M(L) = U(\g,e)/U(\g,e)_\sharp \otimes_{U(\g_0,e)} L,
$$
where $U(\g,e)/U(\g,e)_\sharp$ is viewed as a right
$U(\g_0,e)$-module via the isomorphism $U(\g_0,e) \iso
U(\g,e)_0/U(\g,e)_{0,\sharp}$. We call $M(L)$ the {\em quasi-Verma
module} of type $L$.  In case $L$ is irreducible $M(L)$ is a Verma
modules as defined in \cite[\S 4.2]{BGK}; we consider the more
general situation of quasi-Verma modules in \S\ref{ss:verma} below.

Note that we have the natural inclusion $L \into M(L)$ allowing us
to view $L \sub M(L)$.  Also if $L'$ is a $U(\g_0,e)$-submodule of
$L$, then $M(L')$ is a $U(\g,e)$-submodule of $M(L)$; and we have an
isomorphism $M(L/L') \iso M(L)/M(L')$.  Thus we obtain the following
elementary lemma.

\begin{Lemma} \label{L:quasifilt}
Let $L_1,\dots,L_m$ be the composition factors of $L$. Then the
quasi-Verma module $M(L)$ has a filtration with quotients isomorphic
to $M(L_i)$ for $i = 1,\dots,m$.
\end{Lemma}

\smallskip

Let $M$ be a $U(\g,e)$-module. Given a weight $\lambda \in (\t^e)^*$
we define the $\lambda$-weight space of $M$ to be $M_\lambda = \{m
\in M \mid tm = \lambda(t)m \text{ for all } t \in \t^e\}$. We note
that this labelling of weight spaces differs from that in in
\cite[\S4.2]{BGK}, where the labelling is shifted by
\begin{equation} \label{e:delta}
\delta = \sum_{\substack{1 \leq i \leq r \\
\beta_i|_{\t^e}\in \Phi^e_-\\ d_i \geq 2}}\beta_i +
{\textstyle\frac{1}{2}}\sum_{\substack{1 \leq i \leq r \\
\beta_i|_{\t^e} \in \Phi^e_-
\\
d_i =1}} \beta_i \in (\t^e)^*.
\end{equation}
This shift by $\delta$ is due to the fact that inclusion of $\t^e$
into $U(\g_0,e)$, and the embedding obtained as the composition
$\theta$ with the map in \eqref{e:hwiso} differ.

We say that a weight space $M_\lambda$ of $M$ is a {\em maximal
weight space} if $M_\mu = \{0\}$ for all $\mu > \lambda$.  In this
case, we have $U(\g,e)_\sharp M_\lambda = 0$, so that we obtain an
action of $U(\g_0,e) \iso U(\g,e)_0/U(\g,e)_{0,\sharp}$ on
$M_\lambda$. Suppose $M_\lambda$ is finite dimensional and consider
the induced module $M(L)$. The universal property for quasi-Verma
modules tells us that there is a unique homomorphism
\begin{equation} \label{e:univverma}
M(M_\lambda) \to M
\end{equation}
sending $L \sub M(L)$ identically to $L = M_\lambda \sub M$, see
\cite[Thm.\ 4.5(3)]{BGK}.

\smallskip

The main result in \cite{Go} is a compatibility result between Verma
modules and the loop filtration.  We require this result in
\S\ref{ss:verma}, so we recall it now.  There is a loop filtration
of $U(\g_0,e)$ such that the associated graded algebra $\gr'
U(\g_0,e)$ is naturally isomorphic to $U(\g_0^e)$.  Let $L$ be a
finite dimensional $U(\g_0,e)$-module.  We endow $L$ with the
trivial filtration concentrated in degree $0$, then $L$ is a
filtered module for $U(\g_0,e)$ and the associated graded module
$\gr' L$ is a $U(\g_0^e)$-module.  Note that $x \in \g_0^e(j)$ acts
as zero on $\gr' L$ for $j > 0$.  Thus $\gr' L$ is just the
restriction of $L$ to $U(\g_0^e(0))$.  For technical reasons,
explained in \cite{Go}, we have to consider the {\em shift}
$S_{\delta}(\gr' L)$ of $\gr' L$, where $\delta$ is as in
\eqref{e:delta}. We define $S_{\delta}(\gr' L)$ to be equal to $\gr'
L$ as a vector space, and the action of $x \in \g^e_0$ is given by
``$xv = S_\delta(x)v$'', where $S_\delta: U(\p_0) \to U(\p_0)$ is
defined in analogy to $S_{-\gamma}$ above; note that $\delta$ is a
character of $\p_0$ by \cite[Lem.\ 4.1]{BGK}.

The choice of positive roots $\Phi_+^e$ leads to the triangular
decomposition $\g^e = \g^e_- \oplus \g_0 \oplus \g_+$. Therefore, we
can define the {\em quasi-Verma module}
$$
M(S_{\delta}(\gr' L)) = U(\g^e) \otimes_{U(\g^e_0 \oplus \g^e_+)}
S_{\delta}(\gr' L)
$$
for $U(\g,e)$, where $S_{\delta}(\gr' L)$ is extended to a module
for $U(\g^e_0 \oplus \g^e_+)$ by letting $\g^e_+$ act trivially.
We can define a filtration on $M = M(L)$ as in \S\ref{ss:loop} with
$M_0 = L$, and the associated graded module $\gr' M(L)$ is a module
for $\gr' U(\g,e) \iso U(\g^e)$.
The main theorem of \cite{Go} says that we have an isomorphism
\begin{equation} \label{e:vermaiso}
\xi_L : \gr' M(L) \iso M(S_{\delta}(\gr' L)).
\end{equation}

\smallskip

We finish this section by giving a more explicit description of the
Verma module $M(L)$, which is needed make some identifications in
the proof of Theorem \ref{T:filtverma} below. Let $x_1,\dots,x_t$ is
a basis of $\g^e$ as at the end of \S\ref{ss:structure}.  So $x_i$
has $\t^e$-weight $\gamma_i \in \Phi^e$ for each $i=1,\dots,t$.  We
assume that the basis is ordered so that: $\gamma_1,\dots,\gamma_s
\in \Phi^e_- = -\Phi^e_+$; $\gamma_{s+1},\dots,\gamma_{s'} = 0$; and
$\gamma_{s'+i} = -\gamma_i$ for $i = 1,\dots,s$.  Recall the linear
map $\Theta$ from \eqref{e:Theta}.  Let $l_1,\dots,l_k$ be a basis
of $L$, which can also be viewed as a basis for $\gr' L$. Then the
PBW theorem for $U(\g,e)$ explained at the end of
\S\ref{ss:structure} implies that the vectors

\begin{equation} \label{e:basisverma}
\Theta(x_1)^{a_1}\dots\Theta(x_s)^{a_s} \otimes l_i
\end{equation}

for $a_1,\dots,a_s \in \Z_{\ge 0}$ and $i = 1,\dots,k$ form a basis
of $M(L)$, see also \cite[Thm.\ 4.5(1)]{BGK}.  The results in
\cite{Go}, tell us that
\begin{equation} \label{e:grvermabasis}
\xi_L(\gr' \Theta(x_1)^{a_1}\dots\Theta(x_s)^{a_s} \otimes l_i) =
x_1^{a_1}\dots x_s^{a_s} \otimes l_i
\end{equation}
where $\xi_L$ is as in \eqref{e:vermaiso}.

\subsection{The category $\cO(e)$} \label{ss:catO}

As in \cite[\S4.4]{BGK}, we define $\mathcal O(e) = \mathcal
O(e;\t,\q)$ to be the category of all (finitely generated)
$U(\g,e)$-modules $M$ such that:
\begin{enumerate}
\item[(i)] the action of $\t^e$ on $M$ is semisimple with
finite dimensional $\t^e$-weight spaces; and
\item[(ii)] the set $\{\lambda \in (\t^e)^* \mid M_\lambda \neq 0\}$
is contained in a finite union of sets of the form
$\{\nu\in(\t^e)^* \mid  \nu\leq\mu\}$ for $\mu \in (\t^e)^*$.
\end{enumerate}
As explained in \cite[\S4.4]{BGK} the category $\cO(e)$ depends (in
an essential way) on the choice of positive roots $\Phi^e_+$.

Let $M \in \cO(e)$ and let $V$ be a finite dimensional
$U(\g)$-module. In Proposition~\ref{P:transadm} below, we show that
$M \ostar V \in \cO(e)$; meaning that $\cO(e)$ is stable under
translations. It is a consequence of the following lemma regarding
the isomorphism $\chi_{M,V}$ from \S\ref{ss:lift}.  For the lemma we
note that there is an action of $\t^e$ on $M \ostar V$ from the
embedding $\t^e \into U(\g,e)$ and on $M \otimes V$ through the
embedding $\t^e \into U(\g,e)$ and the inclusion $\t^e \into U(\g)$.

\begin{Lemma}  \label{L:chiequ}
The isomorphism $\chi_{M,V} : M \ostar V \to M \otimes V$ is
$\t^e$-equivariant.
\end{Lemma}

\begin{proof}
Let $T$ be a maximal torus of $G$ with Lie algebra $\t$ and let
$T^e$ by the centralizer of $e$ in $T$, so $\Lie T^e = \t^e$. By
\cite[Lem.\ 2.4]{Pr2}, the adjoint action of $\t^e$ on $U(\g,e)$
coincides with the restriction of the differential of the adjoint
action of $T^e$ on $U(\tilde \p)$.  Let $\bv = (v_1,\dots,v_n)$ be
an ordered basis of $V$ as in \S\ref{ss:lift} and let $\alpha_i \in
(\t^e)^*$ be the $\t^e$-weight of $v_i$; we identify $\alpha_i$ with
the corresponding character of $T^e$.  We may exponentiate the
action of $\t^e$ on $V$ to get an action of $T^e$ on $V$; we write
$tv$ for the image of $v \in V$ under $t \in T^e$. Consider the lift
matrix $\bx^0$ from Lemma~\ref{L:inverse}. For $t \in T^e$, set $t
\cdot \bx^0 = (t \cdot x_{ij}^0)$, where $t \cdot x_{ij}^0$ is the
image of $x_{ij}^0$ under the adjoint action of $t$. Using
Lemma~\ref{L:inverse}, we see that the condition for $\chi_{M,V}$ to
be $\t^e$-equivariant is equivalent to the condition $t \cdot \bx^0
= (\alpha_i(t)^{-1}\alpha_j(t) x_{ij}^0)$ for all $t \in T^e$.

We define an action of $T^e$ on $\tilde Q \otimes V$ by $t \cdot (u
\tilde \1 \otimes v) = (t \cdot u)\tilde \1 \otimes tv$, for $t \in
T^e$, $u \in U(\tilde \p)$ and $v \in V$; this action can also be
seen by restricting the action of $U(\tilde \g)$ to $\g$, noting
that this action is locally finite so that we can exponentiate to an
action of $G$, and then restricting to $T^e$. Now it is a
straightforward calculation to check that
\begin{equation} \label{e:mess}
x \cdot (t \cdot (u \tilde \1 \otimes v)) = t \cdot ((t^{-1} \cdot
x) \cdot (u \tilde \1 \otimes v))
\end{equation}
for $t \in T^e$, $x \in \n$, $u \in U(\tilde \p)$ and $v \in V$.  It
follows from Lemma~\ref{L:inverse} that $\sum_{i=1}^n x_{ij} \otimes
v_i \in U(\g,e) \ostar V$.  Now using \eqref{e:mess}, we see that
$\sum_{i=1}^n t \cdot x_{ij} \otimes \alpha_i(t)v_i \in U(\g,e)
\ostar V$ for $t \in T^e$.  From this we see that the matrix
$(\alpha_i(t)^{-1}\alpha_j(t)(t \cdot x^0_{ij}))$ satisfies
conditions (i) and (ii) in Lemma~\ref{L:inverse}.  Now the
uniqueness statement in that lemma completes the proof.
\end{proof}

\begin{Remark}
One can give an alternative proof to Lemma~\ref{L:chiequ}, by
considering $\chi_{M,V} : (\tilde Q \otimes_{U(\g,e)} M) \otimes V
\to M \otimes V$ and showing that it is $\t^e$-equivariant directly,
where the action on the left-hand side is through the embedding
$\theta : \t^e \into U(\tilde \g)$.  This requires use of
Lemma~\ref{L:free}.
\end{Remark}

The above lemma means that the $\t^e$-weight spaces (shifted by
$\delta$) of $M \ostar V$ and $M \otimes V$ are identified via
$\chi_{M,V}$. It is clear that the weight spaces of $M \otimes V$
satisfy conditions (i) and (ii) in the definition of $\cO(e)$, from
which the proposition below follows.

\begin{Proposition} \label{P:transadm}
We have that $M \ostar V \in \cO(e)$, so $? \ostar V$ defines an exact
endofunctor of $\cO(e)$.
\end{Proposition}

\subsection{Translation of Verma modules} \label{ss:verma}

In this subsection we consider translations of quasi-Verma modules
for $U(\g,e)$. Throughout $L$ is a finite dimensional
$U(\g_0,e)$-module and $V$ is a finite dimensional $U(\g)$-module.
Our main result is Theorem~\ref{T:filtverma}, which says that a
translation of a quasi-Verma module is filtered by quasi-Verma
modules; a consequence is Corollary~\ref{C:filtverma} giving the
corresponding result for Verma modules.

Our first lemma of this subsection  considers the associated graded
side of translations of quasi-Verma modules.

\begin{Lemma} \label{L:grvermaisos}  There are isomorphisms
$$
\gr' (M(L) \ostar V) \iso \gr' M(L) \otimes V \iso M(S_\delta(\gr'
L)) \otimes V \iso M(S_\delta(\gr' L) \otimes V).
$$
So in particular
there is an isomorphism
\begin{equation} \label{e:sigma}
\sigma : \gr' (M(L) \ostar V) \isoto M(S_\delta(\gr' L) \otimes V)
\end{equation}
\end{Lemma}

\begin{proof}
The first isomorphism is given by Proposition~\ref{P:loop}.  The
second is an immediate consequence of \eqref{e:vermaiso}.

The obvious homomorphism $S_\delta(\gr' L) \otimes V \to
M(S_\delta(\gr' L)) \otimes V$ of $U(\g_0^e \oplus \g_+^e)$-modules
extends to a homomorphism $M(S_\delta(\gr' L) \otimes V) \to
M(S_\delta(\gr' L)) \otimes V$. Now a standard argument, using the
basis of $M(S_\delta(\gr' L)) \otimes V$ consisting of elements of
the form in the right-hand side of \eqref{e:grvermabasis}, shows
that this homomorphism is in fact an isomorphism.  This gives the
third isomorphism in the statement of the lemma.
\end{proof}

We now state and prove the main theorem of this section.

\begin{Theorem} \label{T:filtverma}
Let $L$ be a finite dimensional $U(\g_0,e)$-module, and $V$ a finite
dimensional $U(\g)$-module. There is a filtration of $M(L) \ostar V$
by quasi-Verma modules $M(L_1),\dots,M(L_m)$, where $L_i$ is a
finite dimensional $U(\g_0,e)$-module for $i = 1,\dots,m$.
\end{Theorem}

\begin{proof}
Note that $\t^e$ is in the centre of $U(\g_0,e)$ and $U(\g^e)$.
Therefore, when considering $M(L) \ostar V$ we can reduce to the
case where $\t^e \sub U(\g_0,e)$ acts on $L$ by a weight say
$\lambda_0 \in (\t^e)^*$.  This means that  $L \sub M(L)$ is equal
to $M(L)_{\lambda_0}$.

Let  $x_1,\dots,x_t$ be a basis of $\g^e$ as at the end of
\S\ref{ss:hwrec}; so we have bases of quasi-Verma modules as in
\eqref{e:basisverma}and \eqref{e:grvermabasis}. We decompose $V =
\bigoplus_{i=1}^m V_{\lambda_i}$ as a direct sum of $\t^e$-weight
spaces and assume that $\lambda_1,\dots,\lambda_m \in \Z \Phi^e$ are
ordered so that they are non-increasing with respect to the partial
order determined by $\Phi^e_+$.  Fix an ordered basis $\bv$ of $V$
consisting of $\t$-weight vectors; so each $V_\lambda$ is spanned by
a subset of the vectors in $\bv$. Let $\bx^0$ be the lift matrix
from Lemma \ref{L:inverse}. Then we have the vector space
isomorphism $\psi = \psi_{\bx^0,M(L),\bv} : M(L) \otimes V \isoto
M(L) \ostar V$, which is $\t^e$-equivariant by Lemma~\ref{L:chiequ}.

\smallskip

Before constructing the desired filtration of $M(L) \ostar V$, we
use the isomorphisms in the Lemma~\ref{L:grvermaisos} to construct a
filtration of $\gr'(M(L) \ostar V)$ by modules of the form $M(L'_i)
= U(\g^e) \otimes_{U(\g^e_0 \oplus \g^e_+)} L'_i$, where $L'_i$ is a
finite dimensional $U(\g_0^e)$-module.  Having done this, we
construct a corresponding filtration of $M(L) \ostar V$.

Set $M'_0 = \gr'(M(L) \ostar V)$ and consider the $\t^e$-weight
space of $M'_0$ of weight $\lambda_0 + \lambda_1$; we denote $L'_1 =
(M'_0)_{\lambda_0 + \lambda_1}$. This is a maximal weight subspace
of $\gr'(M(L) \ostar V)$, and we have $\sigma^{-1}(S_\delta(\gr'L)
\otimes V_{\lambda_1}) = L'_1$, where $\sigma : M'_0 \isoto
M(S_\delta(\gr' L) \otimes V)$ is given in \eqref{e:sigma}. Now by
the universal property of Verma modules for $U(\g^e)$, there is a
unique homomorphism $g'_1 : M(L'_1) \to M'_0$ sending $L'_1$
identically to itself.  It is clear that $g'_1$ is injective, so we
can identify $M(L'_1)$ with its image in $M'_0$ and define $M'_1 =
M'_0/M(L'_1)$.

Now suppose inductively that we have defined $M'_1,\dots,M'_{j-1}$,
where $M'_i = M'_{i-1}/M(L'_i)$ and $L'_i = (M'_{i-1})_{\lambda_0 +
\lambda_i}$, for $i = 1,\dots,j-1$.  We note that $M'_i$ is a
quotient of $M'_0$ for $i = 1,\dots,j-1$, and inductively that the
image of $\sigma^{-1}(S_\delta(\gr'L) \otimes V_{\lambda_i})$ in
$M'_i$ is equal to $L'_i$, for $i = 1,\dots,j-1$.

Consider $L'_j = (M'_{j-1})_{\lambda_0 + \lambda_j}$.  Using the
basis of $M'_0$ given by elements of the form in
\eqref{e:grvermabasis}, dimension counting of $\t^e$-weight spaces,
and that the weights $\lambda_1,\dots,\lambda_m$ are ordered so that
they are non-increasing, we see that $L'_j$ is a maximal weight
space of $M'_{j-1}$. Further, we see that $L'_j$ is the image of
$\sigma^{-1}(S_\delta(\gr'L) \otimes V_{\lambda_j})$ in $M'_{j-1}$.
We have a unique homomorphism $g'_j : M(L'_j) \to M'_{j-1}$ sending
$L'_j$ identically to itself. A standard PBW argument using the
basis elements of $M'_0$ in \eqref{e:grvermabasis}, shows that
$g'_j$ is injective. Thus, we can view $M(L'_j)$ as a submodule of
$M'_{j-1}$ and define $M'_j = M'_{j-1}/M(L'_j)$.

Thus, inductively we construct $M'_1,\dots,M'_m$.  Counting
dimensions of $\t^e$-weights spaces tells us that $g'_m : M(L'_m)
\to M'_{m-1}$ must be an isomorphism.  Hence, we have constructed a
filtration of $M'_0$ by the quasi-Verma modules $M(L'_i)$ for $i =
1,\dots,m$.

\smallskip

We move on to construct the desired filtration of $M_0 = M(L) \ostar
V$.  Inductively we construct  $M_1,\dots,M_m$ such that $M_i =
M_{i-1}/M(L_i)$ and $L_i = (M_{i-1})_{\lambda_0 + \lambda_i}$, for
$i = 1,\dots,m$. This is done so that we have canonical isomorphisms
$\gr' M_i \iso M'_i$ and $\gr' L_i \iso L'_i$ for $i = 1,\dots,m$.
The construction follows the same lines as above for $M'_0$. We
include the case $j=1$ explicitly, to demonstrate the idea, even
though this is not needed for the induction; this is also the case
for the above construction in $M'_0$.

Consider $L_1 = M(L)_{\lambda_0 + \lambda_1}$; this is equal to
$\psi(L \otimes V_{\lambda_1})$.  Then $L_1$ is a maximal weight
space of $M(L)$ and we clearly have $\gr' L_1 \iso L'_1$; for
example since the loop filtration is invariant under the adjoint
action of $\t^e$. Thus, by the universal property of Verma modules
for $U(\g,e)$, there is a homomorphism $g_1 : M(L_1) \to M_0$ as in
\eqref{e:univverma}. Thanks to \eqref{e:vermaiso} and
\eqref{e:grvermabasis}, we see that $\gr' g_1 : \gr' M(L_1) \to \gr'
M_0$ identifies with $g'_1 : M(L'_1) \to M'_0$.  Thus as $g'_1$ is
injective, a standard filtration argument tells us that $g_1$ is
injective.  We identify $M(L_1)$ with its image in $M_0$ and define
$M_1 = M_0/M(L_1)$. Then we have a canonical isomorphism $\gr' M_1
\iso M'_1$.

Now suppose inductively that we have defined $M_1,\dots,M_{j-1}$,
where $M_i = M_{i-1}/M(L_i)$ and $L_i = (M_{i-1})_{\lambda_0 +
\lambda_i}$, for $i = 1,\dots,j-1$.  Further, we assume inductively
that there are canonical isomorphisms $\gr' M_i \iso M'_i$ and $\gr'
L_i \iso L'_i$, for $i = 1,\dots,j-1$.

Consider $L_j = (M_{j-1})_{\lambda_0 + \lambda_j}$. Through the
isomorphism $\gr' M_{j-1} \iso M_{j-1}$, we get $\gr' L_j \iso L_j$
and thus that $L_j$ is a maximal weight space of $M_{j-1}$. We note
also that $L_j$ is the image of $\psi(L \otimes V_{\lambda_j})$ in
$M_{j-1}$. There is a homomorphism $g_j : M(L_j) \to M_{j-1}$ and,
similarly to the $j=0$ situation, we see that $\gr' g_j : \gr'
M(L_j) \to \gr' M_{j-1}$ identifies with $g'_j : M(L'_j) \to
M'_{j-1}$; here we use the identifications $\gr' M_{j-1} \iso
M'_{j-1}$ and $\gr' M(L_j) \iso M(L'_j)$ from \eqref{e:vermaiso} and
we require \eqref{e:grvermabasis} to see that these maps are equal
under these identifications. Thus, as $g'_j$ is injective, so is
$g_j$. Therefore, we may identify $M(L_j)$ with a submodule of
$M_{j-1}$ and set $M_j = M_{j-1}/M(L_j)$.

Hence, the induction is complete, and we have constructed a
filtration of $M(L) \ostar V$ by quasi-Verma modules.
\end{proof}

The following corollary is an immediate consequence of
Theorem~\ref{T:filtverma} and Lemma~\ref{L:quasifilt}.

\begin{Corollary} \label{C:filtverma}
Let $L$ be a finite dimensional irreducible $U(\g_0,e)$-module, and
$V$ a finite dimensional $U(\g)$-module. Then there is a filtration
of $M(L) \ostar V$ by Verma modules $M(L_1),\dots,M(L_{m'})$, where
$L_i$ is a finite dimensional irreducible $U(\g_0,e)$-module for $i
= 1,\dots,m'$.
\end{Corollary}

\begin{Remark}
In future work, we hope to be more explicit regarding the
isomorphism type of factors that occur in the filtration from
Theorem~\ref{T:filtverma}.  In particular, this should be possible
in the case where $e$ is regular in $\g_0$.  In this case
$U(\g_0,e)$ is isomorphic to $Z(\g_0)$ (the centre of $U(\g_0)$), by
\cite[$\S$2]{Ko}, so we have an explicit parameterization of the
irreducible $U(\g_0,e)$-modules.

We note here that the proof of Theorem~\ref{T:filtverma} does give
the weight of $\t^e$ on each $L_i$ occurring in the filtration.
These are precisely the weights $\lambda_0 + \lambda_i$ for $i =
1,\dots,m$.  Further, an explicit description of the $L_i$ as
subquotients of $M(L) \ostar V$ is given via the map $\psi$.
Finally, we remark that considering the partial order on the weights
$\lambda_1,\dots,\lambda_m$ more carefully allows one to shorten the
filtration, so that the quotients are certain direct sums of the
$M(L_i)$.
\end{Remark}

\section{BRST definition} \label{S:brst}

In this section we define translation in terms of the definition of
the finite $W$-algebra via BRST cohomology.  This turns out to be
key for the proof of Theorem~\ref{T:transdual}.

\subsection{$W$-algebra} \label{ss:brstw}

We recall the definition of the finite $W$-algebra via BRST
cohomology.  The BRST definition is shown to be equivalent to the
Whittaker model definition in \cite{DDDHK}; below we recall some
results from \cite[\S2]{BGK} explaining its equivalence with the
definition of $U(\g,e)$.

We recall that $\tilde{\g} = \g \oplus \k^\ne$ and that $\n^{\ch}$
is a copy of $\n$.  The nonlinear Lie superalgebra $\hat \g$ is
defined to be
$$
\hat{\g} = \tilde{\g} \oplus \n^* \oplus \n^{\ch}
$$
with even part equal to $\tilde{\g}$ and odd part equal to $\n^*
\oplus \n^{\ch}$.  The non-linear Lie bracket $[\cdot,\cdot]$ on
$\hat \g$ is defined by: extending the bracket on $\tilde{\g}$;
declaring that the bracket is identically zero on $\n^*$, $\n^{\ch}$
and between elements of $\tilde \g$ and $\n^* \oplus \n^{\ch}$; and
setting $[f,x^{\ch}] = \lan f,x\ran$ for $f \in \n^*$ and $x \in
\n$, where $\lan f,x \ran$ denotes the natural pairing of $f \in
\n^*$ with $x \in \n$. We have the subalgebra
$$
\hat{\p} = \tilde{\p} \oplus \n^* \oplus \n^{\ch}
$$
of $\hat{\g}$.

Since $\tilde \g$ commutes with $\n^{\ch} \oplus \n^*$, we have the
tensor product decomposition
$$
U(\hat \g) = U(\tilde \g)\otimes U(\n^{\ch} \oplus \n^*).
$$
of $U(\hat \g)$ as a superalgebra.  The factor $U(\n^{\ch} \oplus
\n^*)$ is isomorphic to $\wedge (\n^{\ch}) \otimes \wedge(\n^*)$ as
a vector space with multiplication making it isomorphic to the
Clifford algebra on the space $\n \oplus \n^*$; here
$\wedge(\n^{\ch})$ and $\wedge(\n^*)$ denote the exterior algebras
of $\n^{\ch}$ and $\n^*$ respectively.

We put the {\em charge grading} on $\hat{\g}$, hence also on
$\hat{\p}$, consistent with the $\Z_2$-grading, by declaring that
elements of $\tilde{\g}$ are in degree $0$, elements of $\n^*$ are
in degree $1$, and elements of $\n^{\ch}$ are in degree $-1$. This
induces gradings
$$
U(\hat{\g}) = \bigoplus_{i \in \Z} U(\hat{\g})^i \quad \text{and}
\quad U(\hat \p) = \bigoplus_{i \in \Z} U(\hat{\p})^i.
$$
We note that the charge grading is called the cohomological grading
in \cite[\S2.3]{BGK}.

We recall that $b_1,\dots,b_r$ is a basis of $\n$ and
$f_1,\dots,f_r$ is the dual basis for $\n^*$. Let $d: U(\hat{\g})
\rightarrow U(\hat{\g})$ be the superderivation of charge degree $1$
defined by taking the supercommutator with the degree one element
$$
\delta = \sum_{i=1}^r f_i (b_i - \chi(b_i) - b_i^\ne) -
{\textstyle\frac{1}{2}} \sum_{i,j=1}^r f_i f_j [b_i, b_j]^{\ch}.
$$
One can check that the supercommutator $[\delta,\delta] = 0$, which
means that $d = \ad \delta$ satisfies $d^2 = 0$; we note here that
$[\delta,\delta] = 2\delta^2$, so that $\delta^2 = 0$. This can be
done by computing the action of $d$ on generators of $U(\hat \g)$
explicitly as in \cite{DDDHK}, see \cite[\S2.3]{BGK} for these in
the present notation.  Therefore, we can take the cohomology
$H^0(U(\hat \g),d)$ and we obtain an algebra; this is the {\em BRST
definition of the finite $W$-algebra}.

We now recall the relationship between $H^0(U(\hat \g),d)$ and
$U(\g,e)$ from \cite[\S2.3]{BGK}. First, we recall based on
\cite{DDDHK}, the quasi-isomorphism between $U(\hat \g)$ and the
standard complex $\tilde Q \otimes \wedge(\n^*)$ for computing the
$\n$-cohomology for the dot action of $\n$ on $\tilde Q$.  By the
PBW theorem for $U(\hat \g)$ we have $U(\hat \g) = (U(\tilde \p)
\otimes \wedge(\n^*)) \oplus \hat I$, where $\hat I$ is the left
ideal of $U(\hat \g)$ generated by $\tilde I \sub U(\tilde \g)$ and
$\n^{{\ch}}$.    The map
\begin{equation} \label{e:pqiso}
p : U(\hat \g) \to \tilde Q \otimes \wedge(\n^*)
\end{equation}
is defined to be the projection along this decomposition composed
with the isomorphism $U(\tilde \p) \isoto \tilde Q$. It is a
consequence of results in \cite{DDDHK} that $p$ is a
quasi-isomorphism of complexes.  We write $q$ for $p$ restricted to
$U(\hat \p)^0$ composed with the isomorphism $\tilde Q \iso U(\tilde
\p)$.

The definition of the function $\phi$ in the next lemma is based on
a construction of Arakawa \cite[\S4.8]{Ar} in the case that $e$ is
regular nilpotent.  The following lemma, which says that $\phi$ is a
right-inverse of $q$ is part of \cite[Lem.\ 2.6]{BGK}.

\begin{Lemma} \label{L:tidy}
There is a well-defined algebra homomorphism
\begin{equation} \label{e:phi}
\phi: U(\tilde{\p}) \hookrightarrow U(\hat{\p})^0 z
\end{equation}
such that
$$
\phi(x) = x + \sum_{i=1}^r f_i [b_i, x]^{\ch} \quad \text{ and } \quad
\phi(y^\ne) = y^\ne,
$$
for $x \in \p$ and $y \in \k$. Moreover, $q \phi =
\id_{U(\tilde{\p})}$;
\end{Lemma}

Next we recall \cite[Lem.\ 2.7]{BGK}.

\begin{Lemma}\label{L:mess}
For $u \in U(\tilde{\p})$, we have that
$$
d(\phi(u)) = \textstyle\sum_{i} f_i \phi(\Pr ([b_i - b_i^\ne, u]))
$$
\end{Lemma}

The above lemma is important for proving \cite[Thm.\ 2.8]{BGK}, which says that
$U(\g,e)$ is isomorphic to a certain subalgebra of $U(\hat \g)^0$ isomorphic
to $H^0(U(\hat \g),d)$; we recall this below.

\begin{Theorem}\label{T:tidish}
We have that
$$
U(\g,e) = \{u \in U(\tilde{\p}) \mid d(\phi(u)) = 0\}.
$$
Moreover, we have that $\ker d = \phi(U(\g,e)) \oplus \im\ d$.
\end{Theorem}

\subsection{Translation} \label{ss:brsttrans}

In this subsection, we define translation of $U(\g,e)$-modules in
the BRST definition using the map $\phi$ from \eqref{e:phi}, see
Definition~\ref{D:BRSTtrans} below.  Before giving this definition
we need to recall and introduce some terminology.  When we speak of
graded $U(\hat \g)$-modules below, we always means with respect to
the charge grading.  Throughout this subsection, $M$ is a
$U(\g,e)$-module and $V$ is a finite dimensional $U(\g)$-module.

We recall that a {\em differential graded module} for $U(\hat \g)$
is a graded $U(\hat \g)$-module $N = \bigoplus_{j \in \Z} N^j$ for
$U(\hat \g)$ with a {\em differential} $d_N : N \to N$ such that,
$d_N : N^j \to N^{j+1}$ for each $j$, $d_N^2 = 0$ and
$$
d_N(un) = d(u)n + (-1)^{p(u)}ud_N(n)
$$
for all homogeneous $u \in U(\hat \g)$ and $n \in N$.  For a
differential graded module $N$, each of the cohomology groups
$H^i(N,d_N)$ is a module for $H^0(U(\hat \g),d)$ with the obvious
action.

\begin{Example}
Given a graded $U(\hat \g)$-module $N$. One can define a
differential $d_N : N \to N$ by $d_N(n) = \delta n$: one requires
the fact that $\delta^2 = 0$ in $U(\hat \g)$.
However, this is not the differential
that we shall use for the modules that we consider below.
\end{Example}

We define a ``comultiplication'' $\hat \Delta : U(\hat \g) \to
U(\hat \g) \otimes U(\g)$ by extending $\tilde \Delta$ from
\S\ref{ss:nllatrans} and setting:
$$
\hat \Delta(f)  = f \otimes 1, \quad \text{ and }
\quad \hat \Delta(x^{{\ch}}) =
x^{{\ch}} \otimes 1,
$$
for $f \in \n^*$ and $x \in \n$. It is trivial to verify that $\hat
\Delta$ is a superalgebra homomorphism. Thus given a $U(\hat
\g)$-module $N$, we can define the structure of a $U(\hat
\g)$-module on $N \otimes V$.

We can consider $U(\g,e)$ as a subalgebra of $U(\hat \g)$ through
the injective homomorphism $\phi : U(\tilde \p) \to U(\hat \g)$ from
Lemma~\ref{L:tidy}. Therefore, we may define the induced module
$$
\hat M = U(\hat \g) \otimes_{U(\g,e)} M.
$$
The tensor product $\hat M \otimes V$ is then a graded $U(\hat
\g)$-module, where the grading is defined by declaring that $M$ and
$V$ have degree $0$. We can define a right action of $\delta$ on
$\hat M \otimes V$ by $(u \otimes m \otimes v ) \delta = u\delta
\otimes m \otimes v$; this is well defined since $\phi(U(\g,e)) \sub
\ker d$.  It is straightforward to check that this right action of
$\delta$ commutes with the left action of $U(\hat \g)$. Thus, we
define $d_{M,V}: \hat M \otimes V \to \hat M \otimes V$ by taking
the supercommutator with $\delta$:
$$
d_{M,V}(u \otimes m \otimes v) = \delta (u \otimes m \otimes v) -
(-1)^{p(u)}(u \otimes m \otimes v) \delta.
$$
An elementary commutator calculation gives
$$
d_{M,V}(a  (u \otimes m \otimes v)) = d(a) (u \otimes m \otimes v) +
(-1)^{p(a)} a \cdot d_{M,V}(u \otimes m \otimes v),
$$
for $a, u \in U(\hat \g)$, $m \in M$ and $v \in V$. To see that
$d_{M,V}^2 = 0$, one uses the fact that $\delta^2 = 0$ in $U(\hat
\g)$.  Therefore, $\hat M \otimes V$ is a differential graded
$U(\hat \g)$-module, which leads to our definition of translation in
this setting.

\begin{Definition} \label{D:BRSTtrans}
The {\em BRST definition of the translation of $M$ by $V$} is
$$
H^0(\hat M \otimes V,d_{M,V}).
$$
We view this as a $U(\g,e)$-module through the map $\phi$ from
Lemma~\ref{L:tidy}.
\end{Definition}

We would like a more explicit description of the action of
$d_{M,V}$, which is given by the following lemma.

\begin{Lemma} \label{L:d_V}
The action of $d_{M,V}$ is given by:
$$
d_{M,V}(u \otimes m \otimes v) = d(u) \otimes m \otimes v +
\sum_{i=1}^r f_i u \otimes m \otimes b_i v,
$$
for $u \in U(\hat \g)$, $m \in M$ and $v \in V$.
\end{Lemma}

\begin{proof}
This follows from the fact that
$$
\delta \cdot (u \otimes m \otimes v) = \delta u \otimes m \otimes v
+ \sum_{i=1}^r f_i u \otimes m \otimes b_i v,
$$
which is easily seen from the definitions.
\end{proof}

Next we relate this BRST definition of translation to that given in
Definition~\ref{D:nllatrans}.  To abbreviate notation we write
$$
\tilde M = \tilde Q \otimes_{U(\g,e)} M.
$$
We show that the $U(\g,e)$-modules $H^0(\hat M \otimes V,
d_{M,V})$ and $M \ostar V = H^0(\n,\tilde M \otimes V)$ are isomorphic
by arguing that $\hat M \otimes V$ is quasi-isomorphic to the
standard complex $\tilde M \otimes V \otimes \wedge(\n^*)$ for
calculating $\n$-cohomology for the dot action.

We define
$$
p_{M,V}: \hat M \otimes V \to \tilde M \otimes V \otimes
\wedge(\n^*)
$$
by
$$
p_{M,V}(u \otimes m \otimes v) = p(u)\tilde \1 \otimes m \otimes v,
$$
for $u \in U(\hat \g)$, $m \in M$ and $v \in V$, where $p$ is as in
\eqref{e:pqiso}. Now emulating the arguments in \cite{DDDHK}, one
can show that $p_{M,V}$ is a quasi-isomorphism. For this we require
that $H^i(\n,\tilde M \otimes V) = 0$ for $i \neq 0$, which is a
consequence of Corollary~\ref{C:vanish}. Thus we obtain the desired
result.

\begin{Proposition} \label{P:brstequiv}
We have that $H^0(\hat M \otimes V, d_{M,V}) \iso H^0(\n,\tilde M
\otimes V)$.
\end{Proposition}

We would like to have a ``quasi-inverse'' to $p_{M,V}$.  We define
$$
\phi_{M,V}: \tilde M \otimes V \to \hat M^0 \otimes V
$$
by
$$\phi_{M,V}(u\tilde \1 \otimes m \otimes v) = \phi(u) \otimes m
\otimes v,
$$
for $u \in U(\tilde \p)$, $m \in M$ and $v \in V$. We have $p_{M,V}
\phi_{M,V} = \id_{\tilde M \otimes V}$ so $\phi_{M,V}$ is injective.

The following lemma says that the maps $\phi$ and $\phi_{M,V}$
preserve module structure.

\begin{Lemma}
$$
\phi_{M,V}(a(u\tilde \1 \otimes m \otimes v)) =
\phi(a) \phi_{M,V}(u\tilde \1 \otimes m \otimes v),
$$
for $a, u \in U(\tilde \p)$, $m \in M$ and $v \in V$.
\end{Lemma}

\begin{proof}
We work by induction on the length of $a$ in a PBW basis for
$U(\tilde \p)$.  First consider the case that $a = x \in \p$.  We
have
\begin{align*}
\phi_{M,V}(x (u\tilde \1 \otimes m \otimes v)) & = \phi_{M,V}(xu\tilde \1
\otimes m \otimes v + u\tilde \1 \otimes m \otimes xv) \\
& = \phi(xu) \otimes m \otimes v + \phi(u) \otimes m \otimes xv \\
& = \phi(x) (\phi(u) \otimes m \otimes v) \\
& = \phi(x) \phi_{M,V}(u\tilde \1 \otimes m \otimes v).
\end{align*}
The second to last equality follows from the fact that $\phi$ is an
algebra homomorphism and the definitions of $\phi$ and $\hat
\Delta$. The case $a = x \in \k^\ne$ is trivial.

The induction step is straightforward and we omit the details.
\end{proof}

We now prove the following analogue of Lemma~\ref{L:mess}.

\begin{Lemma} \label{L:blah}
$$
d_{M,V}(\phi_{M,V}(u\tilde \1 \otimes m \otimes v)) = \sum_{i=1}^r
f_i \phi_{M,V} ([b_i-b_i^{\ne},u]\tilde \1 \otimes m \otimes v + u
\tilde \1 \otimes m \otimes b_i v),
$$
for $u \in U(\tilde \p)$, $m \in M$ and $v \in V$.
\end{Lemma}

\begin{proof}
\begin{align*}
d_{M,V}(\phi_{M,V}(u\tilde \1 \otimes m \otimes v)) &=
d_{M,V}(\phi(u) \otimes m \otimes v)
\\ &= d(\phi(u))
\otimes m \otimes v + \sum_{i=1}^r f_i \phi(u) \otimes m \otimes  b_iv \\
&= \sum_{i=1}^r f_i \phi(\Pr([b_i-b_i^{\ne},u])) \otimes m \otimes v
+ \sum_{i=1}^r f_i \phi(u)
\otimes m \otimes  b_iv \\
&= \sum_{i=1}^rf_i \phi_{M,V}([b_i-b_i^{\ne},u]\tilde \1 \otimes m
\otimes v + u \otimes m \otimes b_i v).
\end{align*}
The second equality above is given by Lemma~\ref{L:d_V}, and the
third equality follows from Lemma~\ref{L:mess}.
\end{proof}

Lemma~\ref{L:blah} easily implies the following analogue of Theorem
\ref{T:tidish}.

\begin{Theorem} \label{T:brst}
We have that
$$
M \ostar V = \{z \in \tilde M
\otimes V \mid d_{M,V}(\phi_{M,V}(z)) = 0\}
$$
and $\ker d_{M,V} = \phi_{M,V}(M \ostar V) \oplus \im\ d_{M,V}$.
\end{Theorem}

\begin{Remark} \label{R:BRSTshort}
In the proof of Theorem \ref{T:dualizable}, we consider the case
where $M = U(\g,e)$, where we can identify $\tilde M \otimes V$ with
$\tilde Q \otimes V$ and $\hat M \otimes V$ with $U(\hat \g) \otimes
V$.  We write $d_V$ for the differential on $U(\hat \g) \otimes V$,
and apply the above results with this notational convention.
\end{Remark}

\section{Right-handed versions} \label{S:right}

In Section~\ref{S:transdual}, we require a right-handed versions of
certain definitions and results from earlier in the paper.  The
required material is presented below.

\subsection{Right-handed version of $U(\g,e)$ and translation}
\label{ss:right}

There is a right-handed analogue of $U(\g,e)$, which we denote by
$U(\g,e)'$, note that this notation differs from that used in
\cite[\S2]{BGK}.  An isomorphism between $U(\g,e)$ and $U(\g,e)'$ is
given in \cite[Cor.\ 2.9]{BGK}.  We give the definition of
$U(\g,e)'$ below and recall the isomorphism with $U(\g,e)$, before
discussing the right-handed version of translation.

Let $\tilde I'$ be the right ideal of $U(\tilde \g)$ generated by $x
- x^{\ne} -\chi(x)$ for $x \in \n$, and define $\tilde Q' = U(\tilde
\g)/\tilde I'$; we denote $\tilde \1' = 1 + \tilde I'$. We have a
direct sum decomposition $U(\tilde \g) = U(\tilde \p) \oplus \tilde
I'$; we write $\Pr' : U(\tilde \g) \to U(\tilde \p)$ for the
projection along this direct sum decomposition.  The right twisted
adjoint action of $\n$ on $U(\tilde \p)$ is given by $x \cdot u =
\Pr'([x-x^\ne,u])$ and we define the {\em right-handed finite
$W$-algebra}
$$
U(\g,e)' = H^0(\n,U(\tilde \p))
$$
where the cohomology is taken with respect to the right twisted adjoint
action.

We define $\beta \in \t^*$ by $\beta = \sum_{i=1}^r \beta_i \in
\t^*$; recall that $\beta_i$ is the $\t$-weight of $b_i$, and
$b_1,\dots,b_r$ is a basis of $\n$.   Then $\beta$ extends to a
character  of $\p^*$, by \cite[Lemma 2.5]{BGK}, and we can define
the shift automorphism $S_\beta : U(\tilde \p) \to U(\tilde \p)$ by
$S_\beta(x) = x + \beta(x)$ for $x \in \p$ and $S_\beta(y^\ne) =
y^\ne$ for $y \in \k$. By \cite[Cor.\ 2.9]{BGK}, $S_\beta$ restricts
to an isomorphism $U(\g,e)' \isoto U(\g,e)$.

We can define a right-handed version of translation as follows.  Let
$M'$ be a right $U(\g,e)'$-module and let $V'$ be a finite
dimensional $U(\g)$-module.  There is an obvious structure of a left
$U(\g,e)'$-module on $\tilde Q$, so we can form the $U(\tilde
\g)$-module $M' \otimes_{U(\g,e)'} \tilde Q'$.  Now we can give $(M'
\otimes_{U(\g,e)'} \tilde Q') \otimes V'$ the structure of $U(\tilde
\g)$-module through $\tilde \Delta$.  Given any right $U(\tilde
\g)$-module $E$, we can define a {\em right dot action} of $\n$ on
$E$ by
$$
v \cdot x = v(x-\chi(x)-x^{\ne})
$$
for $v \in E$ and $x \in \n$. Thus we may define
$$
M' \ostar' V' = H^0(\n,(M' \otimes_{U(\g,e)'} \tilde Q') \otimes V'),
$$
where cohomology is taken with respect to the right dot action.  Then
$M' \ostar' V'$ is a right $U(\g,e)'$-module.

We also have a right-handed definition of lift matrices analogous to
Definition~\ref{D:lift} as follows.  Let $V'$ be a finite
dimensional right $U(\g)$-module with ordered basis $\bv' =
(v'_1,\dots,v'_n)$ of $\t$-weight vectors.  Denote the
$c$-eigenvalue on $v_i$ by $c_i'$, and assume that $c'_1 \ge \dots
\ge c'_n$.  We define $V'$-block lower unitriangular matrices in the
analogous way to how $V$-block lower unitriangular matrices are
defined. The coefficient functions $b'_{ij} \in U(\g)^*$ of $V'$ are
defined by $v_i u = \sum_{j=1}^n b'_{ij}(u) v_j$. Then we say that a
$V'$-block lower unitriangular matrix $\bx' = (x'_{ij})$ with
entries in $U(\tilde \p)$ is a lift matrix for $\bv'$ if
$$
\Pr'([x-x^{\ne},x'_{ij}]) + \sum_{k=1}^n b'_{kj}(x) x'_{ik}  = 0,
$$
for all $x \in \n$, and $y_{ij} \in F_{c_j - c_i} U(\tilde \p)$.

There is right-handed version of Proposition~\ref{P:lift}, which
says that if $\bx'$ is a lift matrix for $V'$, and $M'$ is a right
$U(\g,e)'$-module then the map $\psi_{\bx',M',V'} : M' \otimes V'
\to M' \ostar' V'$ defined by $\psi_{\bx',M',V'}(m \otimes v'_j) =
\sum_{j=1}^n m \otimes x'_{ij} \tilde \1'  \otimes v'_i$, for $m \in
M$, is a vector space isomorphism.  Further, given any two lift
matrices $\bx'$ and $\by'$ for $\bv'$, there is a $V'$-block lower
unitriangular matrix $\bw'$ with entries in $U(\g,e)'$ such that
$\by' = \bx' \bw'$.

\subsection{Right-handed version of BRST definition} \label{ss:brstright}

Below we outline the right-handed analogues of the results from
\S\ref{ss:brsttrans} that we require in Section~\ref{S:transdual};
we only consider translations of the regular module, as that is all
that is required. We continue to use the notation from the previous
subsection.

There is a right-handed version $\phi' : U(\tilde \p) \to U(\hat
\g)$ of $\phi$ from \eqref{e:phi} defined by $\phi'(x) = x -
\sum_{i=1}^r [b_i,x]^{\ch} f_i$, giving a version of
Theorem~\ref{T:tidish}.  By \cite[Lem.\ 2.6]{BGK}, we have $\phi' =
\phi S_\beta$.
We can view $U(\hat \g) \otimes V'$ as a differential graded right
$U(\hat \g)$-module, with differential $d_{V'}$ given by taking the
right supercommutator with $\delta$. The cohomology $H^0(U(\hat \g)
\otimes V', d_{V'})$ gives the right-handed version of the BRST
definition of translation of the right regular $U(\g,e)'$-module and
the analogue of Proposition~\ref{P:brstequiv} says that this is
isomorphic as a $U(\g,e)'$-module to $U(\g,e)' \ostar' V'$. Finally,
we have a right-handed version of Theorem~\ref{T:brst}, which in
particular says that for $z \in  \tilde Q' \otimes V'$, we have that
$d_{V'}(\phi'_{V'}(z)) = 0$ if and only if $z \in  M \ostar V$,
where $\phi'_{V'} : \tilde Q' \otimes V' \to U(\hat \g) \otimes V$
is defined by $\phi'_{V'}(u\tilde 1' \otimes v) = \phi'(u) \otimes
v$, for $u \in U(\tilde \p)$ and $v \in V$.

\section{Translation commutes with duality} \label{S:transdual}

In this section we prove Theorem~\ref{T:transdual} saying that
translation commutes with duality for modules in $\cO(e)$; it is
proved for the case $\g = \gl_n(\C)$ in \cite[Thm.\ 8.10]{BK2}.
Before we can state and prove this theorem we need to consider the
notions of restricted duals and dualizable modules.

\subsection{Restricted duals}

A special case of following definition is given in
\cite[(5.2)]{BK2}.

\begin{Definition} \label{D:restdual}
Let $M \in \cO(e)$.  We define the {\em restricted dual} $\bar M$ of
$M$ to be the vector space
$$
\bar M = \bigoplus_{\alpha \in (\t^e)^*} M_\alpha^*,
$$
where $M_\alpha^*$ is the normal dual of the finite dimensional
space $M_\alpha$.  Then $\bar M$ is a subspace of the full dual
$M^*$, so we can define a right action of $U(\g,e)$ on $\bar M$ by
$(fu)(m) = f(um)$ for $f \in \bar M$, $u \in U(\g,e)$ and $m \in M$
making $\bar M$ in to a right $U(\g,e)$-module. Further, we view
$\bar M$ as a right $U(\g,e)'$-module through the isomorphism
$S_\beta$ explained in \S\ref{ss:right}.
\end{Definition}

We define the category $\cO(e)' = \cO(e;\t,\q)$
of right $U(\g,e)'$-modules in analogy to
the category $\cO(e)$.  The following lemma is immediate.

\begin{Lemma} \label{L:dual}
Let $M \in \cO(e)$.  Then $\bar M \in \cO(e)'$.
\end{Lemma}

\subsection{Dualizable modules}

The notion of a dualizable finite dimensional $U(\g)$-module $V$ is
defined in \cite[\S8.4]{BK2} for the case $\g = \gl_n(\C)$.  In
\cite[Thm.\ 8.10]{BK2} it is shown that translation commutes with
duality for dualizable modules, and in \cite[Thm.\ 8.12]{BK2} all
finite dimensional $U(\g)$-modules are proved to be dualizable. We
follow the same approach in our proof of Theorem~\ref{T:transdual}
below.

Let $V$ be a finite dimensional $U(\g)$-module.  We use the notation
$\bar V$ to mean $V^*$ viewed as a right $U(\g)$-module in the usual
way. In the definition below we require lift matrices as defined in
Definition~\ref{D:lift}, and \S\ref{ss:right} for the right-handed
version.  We fix an ordered basis $\bv = (v_1,\dots,v_n)$ of $V$ as
in \S\ref{ss:lift} and let $\bv^* = (v^1,\dots,v^n)$ be the dual
basis of $V^*$.

\begin{Definition} \label{D:dualizable}
We say that $V$ is {\em dualizable} if it is possible to find lift
matrices $\bx$ for $\bv$ and $\by$ for $\bv^*$ such that
$S_\beta(\by) = (S_\beta(y_{ij})) = \bx^{-1}$.
\end{Definition}

For our proof that all finite dimensional $U(\g)$-modules are
dualizable in Theorem~\ref{T:dualizable}, we require the structure
of a differential graded superalgebra on  $U(\hat \g) \otimes
\End(V)$, which is verified in Lemma~\ref{L:diffgrad}.

The endomorphism algebra $\End(V) \iso V \otimes \bar V$ of $V$ is a
$U(\g)$-bimodule in the usual way, i.e.\ $(uau')(v) = u(a(u'v))$ for
$u, u' \in U(\g)$, $a \in \End(V)$ and $v \in V$. We set $v_i^j =
v_i \times v^j$, so that $\{v_i^j \mid i,j = 1,\dots,n\}$ is a basis
of $\End(V)$.

We give $U(\hat \g) \otimes \End(V)$ the tensor product structure of
an algebra; so that it is isomorphic to the algebra of $n \times n$
matrices with entries in $U(\hat \g)$. The charge grading on $U(\hat
\g)$ is extended to $U(\hat\g) \otimes \End(V)$ by declaring that
$\End(V)$ is in degree 0. We note that $U(\hat \g) \otimes \End(V)$
has the structure of a $U(\hat \g)$-bimodule using the
comultiplication $\hat \Delta$ given in \S\ref{ss:brsttrans}.
Therefore we can define a differential $d_{\End(V)}$ on $U(\hat \g)
\otimes \End(V)$ by taking the supercommutator with $\delta$, i.e.
$$
d_{\End(V)}(u \otimes a) = \delta(u \otimes a) - (-1)^{p(u)}(u
\otimes a)\delta
$$
for $u \in U(\hat \g)$ and $a \in \End(V)$
Now arguing as for Lemma~\ref{L:d_V}, we obtain the formula
\begin{equation} \label{e:dEnd}
d_{\End(V)}(u \otimes a) = d(u) \otimes a + \sum_{i=1}^n f_i u
\otimes b_i a - (-1)^{p(u)} u f_i \otimes a b_i.
\end{equation}

\begin{Lemma} \label{L:diffgrad}
With the above definitions $U(\hat \g) \otimes \End(V)$ is a
differential graded superalgebra.  Therefore, $H^0(U(\hat \g)
\otimes \End(V),d_{\End(V)})$ is an algebra.
\end{Lemma}

\begin{proof}
We just need to check that $d_{\End(V)}$ is a
superderivation of $U(\hat \g) \otimes \End(V)$, i.e.
$$
d_{\End(V)}((u \otimes a)(u' \otimes a')) = d_{\End(V)}(u \otimes
a)(u' \otimes a') + (-1)^{p(u)} (u \otimes a)d_{\End(V)}(u' \otimes
a').
$$
The left-hand side is equal to
$$
(d(u)u' + (-1)^{p(u)}ud(u')) \otimes aa' + \sum_{i=1}^n f_i uu'
\otimes b_i aa' - (-1)^{p(uu')} \sum_{i=1}^n uu' f_i \otimes aa'b_i
$$
and the right-hand side is equal to
$$
d(u)u' \otimes aa' + \sum_{i=1}^n f_i uu' \otimes b_iaa' -
(-1)^{p(u)} \sum_{i=1}^n u f_i u' \otimes ab_ia' +
$$
$$
(-1)^{p(u)}( ud(u') \otimes aa' + \sum_{i=1}^n u f_i u' \otimes
ab_ia' - (-1)^{p(u')}\sum_{i=1}^n uu'f_i \otimes aa'b_i).
$$
After a cancellation we see that these expressions are equal.
\end{proof}

\begin{Remark}
The algebras $H^0(U(\hat \g) \otimes \End_V, d_{\End(V)})$
may be of independent interest.  For example, one can show that
$H^0(U(\hat \g) \otimes \End_V, d_{\End(V)})$ is a deformation
of $U(\g,e) \otimes \End(V)$.  We
do not require this here, so we omit the details.
\end{Remark}

We are now in a position to prove that all finite dimensional
$U(\g)$-modules are dualizable.  In the proof we use the notational
convention given in Remark~\ref{R:BRSTshort} and its right-handed
analogue; as well as the notation given above.  This theorem was
proved in the case where $\g$ is of type $A$ in \cite[Thm.\
8.13]{BK2} by directly computing enough lift matrices.

\begin{Theorem} \label{T:dualizable}
Let $V$ be a finite dimensional $U(\g)$-module.  Then $V$ is dualizable.
\end{Theorem}

\begin{proof}
In the right action of $\n$ on $\bar V$, we have that $v^1$ is
killed by $\n$. This implies that the space $V^1$ spanned by
$v_1^1,\dots,v_n^1$ is a $U(\n)$-sub-bimodule of $\End(V)$
isomorphic to $V$ as a left module and trivial as a right module.
This means that $U(\hat \g) \otimes V^1$ is stable under the action
of $d_{\End(V)}$.  Moreover, $U(\hat \g) \otimes V^1$  is a
sub-left-module of $U(\hat \g) \otimes \End(V)$ and as such is
isomorphic to $U(\hat \g) \otimes V$ viewed as a $U(\hat \g)$-module
as in \S\ref{ss:brsttrans}.  Further, the differential $d_V$ on
$U(\hat \g) \otimes V$ identifies with the restriction of
$d_{\End(V)}$ through this isomorphism, by Lemma~\ref{L:d_V} and
\eqref{e:dEnd}.

Similarly, $v_n$ is killed by $\n$ (in the left action of $\n$ on
$V$), and this means the space $V_n$ spanned by $v_n^1,\dots,v_n^n$
is a $U(\n)$-sub-bimodule of $\End(V)$ isomorphic to $\bar V$ as a
right module and trivial as a left module.  This means that $U(\hat
\g) \otimes V_n$ is stable under the action of $d_{\End(V)}$ and is
isomorphic to $U(\hat \g) \otimes \bar V$ viewed as a right $U(\hat
\g)$-module as in  \S\ref{ss:brstright}. Also through this
isomorphism the differential $d_{\bar V}$ identifies with the
restriction of $d_{\End(V)}$ on even parts of the grading of $U(\hat
\g) \otimes V$ and with $-d_{\End(V)}$ on odd parts.

Now let $\bx$ and $\by$ be lift matrices for $\bv$ and $\bv^*$
respectively. Then $\sum_{i=1}^n x_{ij} \tilde \1 \otimes v_i \in
\tilde Q \otimes V$ is an invariant for the dot action of $\n$ for
each $j$, by Proposition~\ref{P:lift}; and $\sum_{j=1}^n
y_{ij}\tilde 1' \otimes v^j$ is an invariant for the right dot
action of $\n$ for each $i$, by the right-handed version of
Proposition~\ref{P:lift}. Now by Theorem~\ref{T:brst} we have
$$
d_V\left(\sum_{k=1}^n \phi(x_{kj}) \otimes v_k\right) = 0,
$$
for each $j$; and by the right-handed version of
Theorem~\ref{T:brst} discussed in \S\ref{ss:brstright} we have
$$
d_{\bar V}\left(\sum_{l=1}^n \phi'(y_{il}) \otimes v^l\right) = 0.
$$
for each $i$.

Putting this altogether and recalling that $\phi' = \phi S_\beta$,
by \cite[Lem.\ 2.6]{BGK}, we have
$$
d_{\End(V)}\left(\sum_{k=1}^n \phi(x_{kj}) \otimes v_k^1\right) = 0,
$$
for each $j$, and
$$
d_{\End(V)}\left(\sum_{l=1}^n \phi(S_\beta(y_{il})) \otimes
v_n^l\right) = 0,
$$
for each $i$.

Now if $u \otimes a,u' \otimes a' \in \ker d_{\End(V)}$, then $(u
\otimes a)(u' \otimes a') = uu'\otimes aa' \in \ker d_{\End(V)}$, by
Lemma~\ref{L:diffgrad}.  Thus, the product
$$
\left(\sum_{l=1}^n \phi'(y_{il}) \otimes
v_n^l\right)\left(\sum_{k=1}^n \phi(x_{kj}) \otimes v_k^1\right)
$$
is in the kernel of $d_{\End(V)}$.  Since $\bx$ and $\by$ are
($V$-block)lower unitriangular, this product is precisely
$$
\sum_{k=1}^n  \phi(S_\beta(y_{ik})) \phi(x_{kj}) \otimes v_n^1,
$$
i.e.\ the $(i,j)$ entry of $\phi(S_\beta(\by)\bx)$ tensored with
$v_n^1$. Now since $v_n^1$ is killed by $\n$ on both sides, we must
have that $d(\phi(S_\beta(\by)\bx)_{ij}) = 0$ for each $i$ and $j$.
Therefore, by Theorem~\ref{T:tidish}, we have that
$(S_\beta(\by))\bx)_{ij} \in U(\g,e)$. Thus, $\bx S_\beta(\by) =
\bw$ is a $V$-block lower unitriangular matrix with entries in
$U(\g,e)$. Now replacing $\bx$ with $\bw^{-1} \bx$ and using
Proposition~\ref{P:lift}(c) we see that $V$ is dualizable.
\end{proof}

The next lemma, which generalizes \cite[Lem.\ 8.8]{BK2}, follows
from Proposition~\ref{P:lift}(c) and Theorem~\ref{T:dualizable}.

\begin{Lemma} \label{L:invlift}
Let $\bx$ be a lift matrix for $\bv$, then $S_{-\beta}(\bx)^{-1}$ is
a lift matrix for $\bv^*$.
\end{Lemma}

We now give some notation that we use the next subsection.  Let $M'$
be a right $U(\g,e)'$-module.  Recall the lift matrix $\bx^0$ from
Lemma~\ref{L:inverse}.  Letting $\by^0 = S_{-\beta}(\bx^0)^{-1}$,
the discussion in \S\ref{ss:right} and Lemma~\ref{L:invlift} imply
that $\psi'_{M',\bar V,\by^0} : M' \ostar' \bar V \to M' \otimes
\bar V$ is a vector space isomorphism.  We denote its inverse by
$\chi_{M',\bar V} : M' \otimes \bar V \to M' \ostar' \bar V$.

\subsection{Main theorem}

We now state and prove the main theorem of this section.
The proof is almost identical to that of \cite[Thm.\ 8.10]{BK2}, we
include the details for completeness. In the proof we use the bases
of $V$ and $\bar V$ from the previous subsection, and to simplify
notation we write $\bx = \bx^0$ and $\by = \by^0$ for the lift
matrices.

\begin{Theorem} \label{T:transdual}
Let $M \in \cO_e$ and $V$ be a finite dimensional $U(\g)$-module.
Then there is an isomorphism of $U(\g,e)'$-modules
$$
\bar M \ostar' \bar V \iso \bar{M \ostar V}.
$$
\end{Theorem}

\begin{proof}
We define $\omega_{M,V}: \bar M \ostar' \bar V \to \bar{M \ostar V}$
to be the composite
$$
\begin{CD}
\bar M \ostar' \bar V &@>{\chi_{\bar M,\bar V}}>>& \bar M \otimes
\bar V \iso \bar{M \otimes V} &@>{\bar{\chi_{M,V}}}>>& \bar{M \ostar
V},
\end{CD}
$$
where $\bar{\chi_{M,V}}$ is the dual map of $\chi_{M,V}$. Then
$\omega_{M,V}$ is a vector space isomorphism.

Define $\delta : U(\tilde \p) \to U(\tilde \p) \otimes \End(V)$ to
be the composite $(\id_{U(\tilde \p)} \otimes \rho) \tilde \Delta$,
where $\rho : U(\g) \to \End(V)$ is the representation of $U(\g)$ on
$V$.

So for $u \in U(\g,e)$ and $m \in M$, we have
$$
u\left(\sum_{i=1}^n x_{ij} \tilde \1 \otimes m \otimes v_i \right)
= \sum_{k,i=1}^n u^*_{ik} x_{kj} \tilde \1 \otimes m \otimes v_i \in
M \ostar V,
$$
where $\delta(u) = \sum_{i,j=1}^n u^*_{ij} \otimes v_i^j$. Suppose
for a moment that $M = U(\g,e)$, (it does not matter for this part
of the argument that $U(\g,e)$ does not lie in $\cO(e)$). By
Proposition~\ref{P:lift}, we must have $\sum_{i,k=1}^n u^*_{ik}
x_{kj} \tilde \1 \otimes 1 \otimes v_i = \sum_{i,k=1}^n x_{ik}
\tilde \1 \otimes u_{kj} \otimes v_i$, where $u_{kj} \in U(\g,e)$;
so we have $\sum_{k=1}^n x_{ik}u_{kj} = \sum_{k=1}^n u^*_{ik}
x_{kj}$.  This means that $\bu = S_\beta(\by) \bu^* \bx$, where $\bu
= (u_{ij})$ and $\bu^* = (u^*_{ij})$.

Therefore, for general $M$ we have
$$
u\left(\sum_{i=1}^n x_{ij} \otimes m \otimes v_i\right) =
\sum_{i,k=1}^n x_{ik} \otimes u_{kj} m \otimes v_i.
$$
This means that through the isomorphism $\chi_{M,V}$ the action of
$U(\g,e)$ on $M \otimes V$ is given by
\begin{equation} \label{e:act1}
u(m \otimes v_j) = \sum_{i=1}^n u_{ij}m \otimes  v_i,
\end{equation}
and the $u_{ij}$ are defined from $\bu =S_\beta(\by) \bu^* \bx$.

Now an analogous argument gives the action of $U(\g,e)'$ on $\bar M
\otimes \bar V$ through the isomorphism $\chi_{\bar M,\bar V}$.  We
define $\delta' = (\id_{\End_U(\p)} \otimes \bar \rho) \tilde \Delta
: U(\tilde \p) \to U(\tilde \p) \otimes \End'(\bar V)$, where
$\rho'$ is the right representation of $U(\g)$ mapping into the
space of right endomorphisms $\End'(\bar V)$ of $\bar V$.  Given $u
\in U(\g,e)'$ we define the matrix $\bu'^*$ by $\delta'(u) =
\sum_{i,j=1}^n u'^*_{ij} \otimes v_i^j$. Then the action of
$U(\g,e)'$ is given by
\begin{equation} \label{e:act2}
(f \otimes v^i)u = \sum_{j=1}^n f u'_{ij} \otimes v^j,
\end{equation}
where $\bu' = (u'_{ij})$ is defined by
$\bu' = \by \bu'^*S_{-\beta}(\bx)$, so that $\bu' = S_\beta(\bu)$.

We view $\bar{M \ostar V}$ as a $U(\g,e)'$-module as in Definition
\ref{D:restdual}.  Then \eqref{e:act1} and
\eqref{e:act2} imply that $\omega_{M,V}$
is an isomorphism of $U(\g,e)'$-modules.
\end{proof}

\begin{Remark}
We note that we can weaken the hypothesis that $M \in \cO(e)$: we
just require that $M$ is the direct sum of finite dimensional
generalized $\t^e$-weight spaces.
\end{Remark}

The first part of the following corollary is immediate from the
proof of Theorem~\ref{T:transdual}.  The second part can be verified
by direct calculation.

\begin{Corollary} \label{C:tdprop}
Let $M$ be a $U(\g,e)$-module and $V$, $V'$ be finite dimensional
$U(\g)$-modules.
Then the following diagrams commute:
\begin{enumerate}
\item[(i)]
$$
\begin{CD}
\bar M \ostar' \bar V  &@>{\omega_{M,V}}>> \bar{M \ostar V} \\
@V{\chi_{\bar M,\bar V}}VV& @AA{\bar{\chi_{M,V}}}A\\
\bar M \otimes \bar V &@>{\sim}>> \bar{M \otimes V},
\end{CD}
$$
where the bottom map is the canonical isomorphism;
\item[(ii)]
$$
\minCDarrowwidth60pt
\begin{CD}
(\bar M \ostar' \bar V) \ostar' \bar{V'} &@>{\omega_{M,V} \ostar
\id_{\bar{V'}}}>> (\bar{M \ostar V}) \ostar' \bar{V'}
&@>{\omega_{M\ostar V,V'}}>> \bar{(M \ostar V) \ostar V'}\\
@V{a_{\bar M,\bar V, \bar{V'}}}VV& & &
&@AA{\bar{a_{M,V,V'}}}A \\
\bar M \ostar' (\bar V \otimes \bar V')  &@>{\sim}>> \bar M \ostar'
(\bar{V \otimes V'}) &@>{\omega_{M, V \otimes V'}}>> \bar{M \ostar
(V \otimes V')},
\end{CD}
$$
where the bottom left map is the canonical isomorphism,
$\bar{a_{M,V,V'}}$ is the dual of the isomorphism from
Lemma~\ref{L:assoc}, and $a_{\bar M,\bar V, \bar{V'}}$ is the
right-handed version defined in analogy.
\end{enumerate}
\end{Corollary}

\section{Independence of $\l$ and good grading}
\label{S:indep}

As mentioned in \S\ref{ss:whittw}, the definition of $W_\l$ depends
on the choice of good grading $\g = \bigoplus_{j \in \R} \g(j)$ for
$e$ and choice of isotropic subspace $\l$ of $\g(-1)$. In this
subsection we briefly recall the arguments from \cite{GG} and
\cite{BG} showing that $W_\l$ is independent up to isomorphism of
these choices.  Then we show that through these isomorphisms, the
definition of translation does not depend on the choices of good
grading and $\l$ in the appropriate sense.

\subsection{Independence of choice of $\l$}

We use the notation from \S\ref{ss:whittw} and
\S\ref{ss:whitttrans}.  In particular, $\l$ is an isotropic subspace
of $\k$ used in the definition of $W_\l$.

Suppose $\l'$ is another isotropic subspace of $\g(-1)$ with $\l'
\sub \l$.   Then we have the surjection $Q_{\l'} \onto Q_\l$.  This
restricts to a map $\psi_{\l',\l}: W_{\l'} \to W_{\l}$, which is
shown in \cite[\S5.5]{GG} to be an isomorphism; this is proved by
arguing that the associated graded map is an isomorphism.

Now suppose that $\l'$ is any isotropic subspace of $\g(-1)$.  Take
$\l'' = 0$, then we have isomorphisms $\psi_{\l'',\l}: W_{\l''}
\isoto W_{\l}$ and $\psi_{\l'',\l'}: W_{\l''} \isoto W_{\l'}$.
Therefore, we obtain a canonical isomorphism $\psi_{\l',\l} =
\psi_{\l'',\l} \psi_{\l'',\l'}^{-1} : W_{\l'} \to W_{\l}$ ; note
that there is no ambiguity in our notation.

Let $M$ be a $W_\l$-module and $V$ a finite dimensional
$U(\g)$-module.  In Proposition~\ref{P:indepiso} we show that $M
\ostar_\l V$ is isomorphic to $M \ostar_{\l'} V$ in the appropriate
sense. First we note that we may view $M$ as a $W_{\l'}$-module via
``$u m = \psi_{\l',\l}^{-1}(u) m$'' for $u \in W_{\l'}$ and $m \in
M$; thus we can define $M \ostar_{\l'} V$.

\begin{Proposition} \label{P:indepiso}
There is an canonical isomorphism of vector spaces $\xi_{\l',\l} : M
\ostar_{\l'} V \isoto M \ostar_{\l} V$ such that
$$
\xi_{\l',\l}(uz) = \psi_{\l',\l}(u) \xi_{\l',\l}(z)
$$
for all $u \in W_{\l'}$, $z \in M \ostar_{\l'} V$.
\end{Proposition}

\begin{proof}
It suffices to consider the case $\l' \sub \l$; the general case can
be dealt with by taking a composition as for $\psi_{\l',\l}$.

The surjection $Q_{\l'} \onto Q_\l$ gives a map $Q_{\l'}
\otimes_{W_{\l'}} M \onto Q_\l \otimes_{W_\l} M$. This gives rise to
a homomorphism of $U(\g)$-modules $(Q_{\l'} \otimes_{W_{\l'}} M)
\otimes V \onto (Q_\l \otimes_{W_\l} M) \otimes V$, which in turn
restricts to a map $\xi_{\l',\l}: M \ostar_{\l'} V \to M \ostar_\l
V$.

There are Kazhdan filtrations on $Q_\l$, $Q_{\l'}$, $M$ and $V$ as
defined in \S\ref{ss:whittw} and \S\ref{ss:kazh}; we take the same
filtration of $M$ considered as a module for $W_\l$ and $W_{\l'}$.
Therefore, we have filtrations on $M \ostar_{\l'} V$ and $M
\ostar_\l V$. To show that $\xi_{\l',\l}$ is an isomorphism, it
suffices, by a standard filtration argument, to show that the
associated graded map $\gr \xi_{\l',\l} : \gr (M \ostar_{\l'} V) \to
\gr (M \ostar_\l V)$ is an isomorphism.

As in the proof of Theorem~\ref{T:iso}, we can identify $\gr((Q_\l
\otimes_{W_\l} M) \otimes V) \iso \gr M \otimes \C[N_\l] \otimes V$
and $\gr(M \ostar_\l V)$ with the $N_\l$-invariants $(\gr M \otimes
\C[N_\l] \otimes V)^{N_\l}$.  Under these identifications the map
$\gr M \otimes \C[N_{\l'}] \otimes V \to \gr M \otimes V$ induced
from evaluation at $1$ restricts to an isomorphism $\bar \epsilon :
\gr(M \ostar_\l V) \isoto \gr M \otimes V$ as in
\eqref{e:barepsilon}. Also
$$
\gr \xi_{\l',\l}: (\gr M \otimes
\C[N_{\l'}] \otimes V)^{N_{\l'}} \to (\gr M \otimes \C[N_\l] \otimes
V)^{N_\l}
$$
is induced from the restriction map $\C[N_{\l'}] \to \C[N_\l]$.
Thus, it is straightforward to see that the diagram below commutes,
which implies that $\gr \xi_{\l',\l}$ is an isomorphism as required.
$$
\begin{CD}
(\gr M \otimes \C[N_{\l'}] \otimes V)^{N_{\l'}}
&@>{\gr \xi_{\l',\l}}>>&(\gr M \otimes \C[N_\l] \otimes V)^{N_\l}\\
@V{\bar \epsilon'}VV&&@VV{\bar \epsilon}V\\
\gr M \otimes V &@>{\id}>>& \gr M \otimes V
\end{CD}
$$
It is clear from construction that $\xi_{\l',\l}$ satisfies the
condition in the proposition.
\end{proof}

\subsection{Independence of good grading} \label{ss:indepgood}

We now introduce some notation and terminology required to show that
the definition of $W_\l$ does not depend on the choice of good
grading up to isomorphism.

In this section we allow ourselves to consider the more general
notion of good $\R$-gradings: we recall that an $\R$-grading $\g =
\bigoplus_{j \in \R} \g(j)$ for $e$ is good for $e$ if $e \in
\g(2)$, $\g^e \sub \bigoplus_{j \geq 0} \g(j)$ and $\z(\g) \sub
\g(0)$. Then the alternating form $\lan \cdot | \cdot \ran$ can be
defined on $\k = \g(-1)$ in exactly the same way as for case of
$\Z$-gradings. For an isotropic subspace $\l$ of $\k$ we can define
$\m_\l$, $\n_\l$, $Q_\l$ and $W_\l$ using the same process as in
\S\ref{ss:whittw}. In this subsection we only consider the case
where $\l$ is a Lagrangian subspace of $\g(-1)$ so that $\m_\l =
\n_\l$.

We use the notation $\Gamma : \g = \bigoplus_{j \in \R} \g(j)$ to
denote the good grading for $e$, and let $\Gamma' : \g =
\bigoplus_{j \in \R} \g'(j)$ be another good grading for $e$. Let
$\l'$ be a Lagrangian subspace of $\g'(-1)$.  Then $\m'_{\l'}$,
$Q'_{\l'}$ and $W'_{\l'}$ are defined in analogy to $\m_\l$, $Q_\l$
and $W_\l$.

The proof, of \cite[Thm.\ 1]{BG}, that the definition of $W_\l$ does
not depend on the choice of good grading is based on the notion of
adjacency. We recall from \cite{BG} that the two good gradings $\g =
\bigoplus_{j \in \R} \g(j)$ and $\g = \bigoplus_{j \in \R} \g'(j)$
are called {\em adjacent} if it is possible to choose a Lagrangian
subspaces $\l$ of $\g(-1)$ and $\l'$ of $\g'(-1)$ such that $\m_\l =
\m'_{\l'}$.  In this case we simply have $W_\l = W'_{\l'}$: equality
as algebras and therefore certainly isomorphic.

Next suppose that the two good gradings $\Gamma$ and $\Gamma'$ for
$e$ are conjugate by $g \in G^e$, where $G^e$ is the centralizer of
$e$ in $G$.  Let $\l'$ be the image of $\l$ under the adjoint action
of $g$.  Then it is clear that $g$ induces an isomorphism $\gamma_g
: W_\l \isoto W'_{\l'}$.

Let $\Gamma$ and $\Gamma'$ be two good gradings for $e$. The key
ingredient for the proof that the definition of $W_\l$ does not
depend on the choice of good grading is \cite[Thm.\ 2]{BG}, which
says: there exists a chain $\Gamma_1, \dots, \Gamma_n$ of good
gradings for $e$ such that $\Gamma$ is conjugate to $\Gamma_1$ and
$\Gamma' = \Gamma_n$, and $\Gamma_i$ is adjacent to $\Gamma_{i+1}$
for each $i=1,\dots,n-1$. So we obtain an isomorphism $\phi : W_\l
\isoto W'_{\l'}$, by composing an isomorphism of the form $\gamma_g$
(to move from $\Gamma$ to $\Gamma_1$) with isomorphisms of the form
$\psi_{\l_i,\l_{i+1}}$ (to move between $\Gamma_i$ and
$\Gamma_{i+1}$).

Let $M$ be a $W_\l$-module and $V$ a finite dimensional
$U(\g)$-module.  Theorem~\ref{T:indepgood} below, which says that
``translation does not depend on the choice of good grading''.  In
the statement we write $M'$ for $M$ viewed as an $W'_{\l'}$-module
through $\phi$; and we write $M' \ostar'_{\l'} V$ for translation of
the $W'_{\l'}$-module $m'$ by $V$.  It is proved by taking a chain
$\Gamma_1, \dots, \Gamma_n$ of good gradings as above, and
constructing $\eta$ as composition of isomorphisms. First one takes
an isomorphism $\delta_g$ that is determined by conjugation by $g$
(in a similar way to how $\gamma_g$ is defined), then composes with
isomorphisms of the form $\xi_{\l_i,\l_{i+1}}$ from
Proposition~\ref{P:indepiso}.

We have to observe that the arguments required for the proof
Proposition~\ref{P:indepiso} go through in the more general setting
where we are considering a good $\R$-grading.  In particular, the
filtrations of algebras and modules considered are not necessarily
integral, but indexed by a countable subset $I$ of $\R$. This subset
$I$ is closed under addition and every subset of $I$ has a greatest
and least element with respect to $<$. These conditions mean that
one can make sense of all the usual notions associated to
filtrations, e.g.\ filtered algebras and associated graded modules.

\begin{Theorem} \label{T:indepgood}
There is an isomorphism of vector spaces $\eta : M \ostar_\l V
\isoto M' \ostar_{\l'} V$ such that
$$
\eta(uz) = \phi(u) \eta(z)
$$
for all $u \in W_{\l'}$ and $z \in M \ostar_{\l'} V$.
\end{Theorem}

\end{document}